\documentclass[a4paper,12pt]{article}

\usepackage[utf8]{inputenc}
\usepackage[T1]{fontenc}

\usepackage[english]{babel}

\usepackage{amsmath,amssymb,amsfonts,stmaryrd}
\usepackage{amsthm,amscd,graphicx}

\usepackage{bbm}
\usepackage{dsfont}

\usepackage{authblk}

\usepackage[colorlinks,
            pdftex,
            pdfauthor={},
            pdftitle={},
            pdfsubject={},
            pdfkeywords={}]{hyperref}
            
            \hypersetup{
            	colorlinks=true,
            	linktoc=none
            }

\newtheorem{theorem}{Theorem}[section]

\newtheorem{lemma}[theorem]{Lemma}

\theoremstyle{definition}

\newcommand{\ensnombre}[1]{\mathbb{#1}}
\newcommand{\N}{\ensnombre{N}}

\newcommand{\R}{\ensnombre{R}}

\newcommand{\defeq}{\mathrel{\mathop:}=}

\def \P{\mathbb{P}}

\newcommand{\abs}[1]{\left\lvert #1 \right\rvert}
\newcommand{\norme}[1]{\left\lVert #1 \right\rVert}

\usepackage{a4wide}

\usepackage{tikz}

\title{Posterior contraction rates for constrained deep Gaussian processes in density estimation and classification}

\author[1]{Fran\c{c}ois Bachoc}
\author[2]{Agn\`es Lagnoux}

\affil[1]{Institut de Math\'ematiques de Toulouse; UMR5219. Université de Toulouse; CNRS. UT3, F-31062 Toulouse, France. \hspace{3cm}
	 \url{francois.bachoc@math.univ-toulouse.fr}}
\affil[2]{Institut de Math\'ematiques de Toulouse; UMR5219. Université de Toulouse; CNRS. UT2J, F-31058 Toulouse, France.  \hspace{3cm}
   \url{lagnoux@univ-tlse2.fr}}

\begin{document}

\date{13 december 2021}

\maketitle

\begin{abstract}
We provide posterior contraction rates for constrained deep Gaussian processes in non-parametric density estimation and classification. The constraints are in the form of bounds on the values and on the  derivatives of the Gaussian processes in the layers of the composition structure.  The contraction rates are first given in a general framework, in terms of a new concentration function that we introduce and that takes the constraints into account. Then, the general framework is applied to integrated Brownian motions, Riemann-Liouville processes, and Mat\'ern processes and to standard smoothness classes of functions. In each of these examples, we can recover known minimax rates.
\end{abstract}




\section{Introduction}

Gaussian processes are widely used in statistics and machine learning to model a wide range of data \cite{rasmussen06gaussian}. They have become a popular method for a
large range of applications, such as geostatistics \cite{bevilacqua2012estimating,matheron70theorie,porcu2016spatio}, numerical code approximation \cite{sacks89design,santner03design,bachoc16improvement}, calibration \cite{paulo12calibration,bachoc14calibration}, and global optimization \cite{jones98efficient}.
 In particular, they have been adopted in nonparametric Bayesian methods for constructing prior distributions for infinite-dimensional statistical models in several settings. Many related theoretical contributions were obtained since the late 2000's \cite{van2008rates} and practical contributions have been developed even before \cite{lenk1988logistic}.
We refer to the following books on the topic
\cite{rasmussen06gaussian, ghosal2017fundamentals} and the references therein. 
For instance, when considering the nonparametric estimation of a density relative to some measure $\mu$ from a sample $(X_1,\dots, X_n)$, one may consider as prior distribution the random density
\[
p_Z(x)=\frac{e^{Z(x)}}{\int e^{Z(y)}d\mu(y)},
\]
where $(Z(x))_{x\in \mathcal X}$ is a Gaussian process indexed by the space $\mathcal X$ of the observations. Here the exponential form forces the prior to weight only nonnegative functions and the renormalization permits to integrate to unity. We refer to Section \ref{ssec:density} for more details on this context and also to, e.g., \cite{rasmussen06gaussian,lenk1991towards,tokdar2007posterior,ghosal2006posterior}. 
Dealing with classification is also of interest and can be handled again using Gaussian processes. Here the estimation of the binary regression function
$\P(Y=1\vert X=x)$ 
can be done from a bivariate sample $((X_1,Y_1),\dots, (X_n,Y_n))$ and using priors of the form 
\[
f_Z(x)=\Psi(Z(x)),
\]
where $\Psi$ is a fixed measurable function from $\mathbb{R}$ to $(0,1)$.   
We refer to Section \ref{ssec:classif} and \cite{choudhuri2007nonparametric,wood1998bayesian} for more details on this context.
Other settings have also been considered in the literature as
 regression with fixed covariates \cite{van2008rates,choi2007posterior}, and white noise models  \cite{van2008rates} to name a few.
In the literature, several examples of Gaussian priors have been considered as the integrated Brownian motion \cite{van2008rates}, the Riemann-Liouville process \cite{van2008rates,castillo2008lower}, the Matérn process \cite{van2011information}, and the  exponential process \cite{Yin1991}. 

Now given a prior as illustrated in the previous paragraph and the observations, Bayes' rule leads to a posterior distribution on the function of interest defining the unknown data distribution. One is then typically interested in proving posterior consistency, meaning that the posterior distribution converges to this function of interest as the
sample size goes to infinity (see, e.g., \cite{barron1999consistency}).
Another question of interest is the one of rates of contraction of posterior distributions based on Gaussian process priors. In a nutshell, the rate of contraction of the posterior corresponds to an $\varepsilon_n$ as small as possible such that the posterior probability of the ball centered at the function of interest and of radius $\varepsilon_n$ still converges to one in probability. Other directions have been investigated. In particular, while upper bounds are avalaible and express in terms of a concentration function involving the reproducing kernel Hilbert space (RKHS) associated to the Gaussian prior, the author of \cite{castillo2008lower} exhibits lower bound counterparts. 
More recently, more flexibility has been allowed considering a randomly rescaled smooth Gaussian field as a prior that adapts to smoothness \cite{van2009adaptive}.

In this article, we consider deep Gaussian priors to reach further generality, similarly as deep neural networks
are exploited to go beyond standard shallow neural networks. The reader is referred to \cite{lee2017deep,garriga2018deep,matthews2018gaussian,dunlop2018deep}
for some references on both deep Gaussian processes and deep neural networks. 
Deep Gaussian processes, introduced in \cite{DL13}, are non-Gaussian stochastic processes, constructed from a network of Gaussian processes, similarly as neural networks. 
A deep Gaussian process is then a stochastic process from $\mathbb{R}^d$ to $\mathbb{R}$ of the form
\begin{equation} \label{eq:general:DGP}
Z_H \circ Z_{H-1} \circ \dots \circ Z_1, 
\end{equation}
where, for $h = 1 , \dots, H$, $Z_h$ is a (multivariate) Gaussian process from $\mathbb{R}^{d_{h}}$ to $\mathbb{R}^{d_{h+1}}$, with the convention $d_1 = d$ and $d_{H+1} = 1$.  In analogy to neural networks,  $Z_2,\dots,Z_{H-1}$ may be interpreted as hidden layer processes and $H$ may be called the number of layers.
In particular, since deep Gaussian processes are based on composing Gaussian processes, their mathematical analysis becomes challenging.

Deep Gaussian processes are commonly used as a prior for a function that is observed exactly. For instance, \cite{hebbal2021bayesian} tackles optimization and \cite{radaideh2020surrogate} deals with surrogate modeling of computer models with applications to nuclear engineering. 
Deep Gaussian processes are exploited as well with noisy/indirect function observations \cite{damianou2013deep,salimbeni2017doubly}. 
From a methodological point of view, deep Gaussian processes are also an efficient way to obtain more non-stationarity than simple Gaussian processes. This need for non-stationarity is largely acknowledged in the Gaussian process literature (see e.g.\  \cite{gramacy2008bayesian,marmin2018warped}). Hence, because of its practical impact, the theoretical analysis of deep Gaussian processes is beneficial.

In this work, several Gaussian processes are constrained (in terms of bounded norms for the processes themselves and their first derivatives) and then composed 
to form constrained deep Gaussian processes. The constraints are necessary for our theoretical analysis (see more details in Section \ref{ssec:setting} in the discussion of \eqref{eq:first:constraint} and \eqref{eq:second:constraint}) and may be useful in practice to exclude very irregular realizations from the Bayesian model.
Note that these bounds and derivative constraints may also be enforced to standard Gaussian processes, for a better modeling in some situations, especially with additional available expertise \cite{Golchi2015MonotoneEmulation,LopezLopera2017FiniteGPlinear,maatouk2017gaussian}. 
 In the same vein of previous works, we investigate posterior contraction for density estimation and classification. We establish original general rates of contraction extending those available for classical Gaussian processes (see Theorems \ref{th:Th3.1} and \ref{th:Th3.2}). The proofs are original and exploit the results established in \cite{van2008rates}.
In \cite{van2008rates}, the authors consider a single  Gaussian process defined on a compact space and valued in $\R$. A prior step before  proving our results is then the construction of a single global Gaussian process from the collection of the Gaussian processes involved in the deep Gaussian process prior (see Appendix \ref{app:inter}). 

In addition, we study several examples of deep Gaussian priors: integrated Brownian motions (Section \ref{ssec:integrated:brownian}) and Riemann-Liouville processes (Section \ref{ssec:riemann}) in dimension one, and Matérn processes in general dimension (Section \ref{ssec:matern}). It appears that the optimal rates are recovered in these examples for stantard classes of functions. The proofs of these results rely on the proofs of the analog results in the context of classical (single) Gaussian processes \cite{van2008rates,van2011information}.

We would like to mention that, while preparing this article, we have been aware of the independent work of Finocchio and Schmidt-Hieber \cite{finocchio2021posterior} on the same topic. Their paper presents many interesting results. In contrast to our work, they consider the problem of regression rather than density estimation and classification. They address adaptivity with respect to the smoothness and structure of the function of interest, which we do not. Their proofs are independent of the proofs of the present paper and the techniques used  are also different. To our knowledge, \cite{finocchio2021posterior} is the only already existing work providing posterior contraction results for deep Gaussian processes.

The paper is organized as follows. In Section \ref{sec:setting}, we present the setting and some notation. Section \ref{sec:posterior_contraction} is dedicated to  posterior contraction for both density estimation and classification. Examples of rates of contraction for specific function classes and priors are given in Section \ref{sec:examples}. 
 Appendix \ref{appendix:equivalent:constraints} explains how to allow for more flexibility on the constraints considered in Section \ref{sec:setting} using linear transformations of inputs and outputs. In Appendix \ref{app:inter}, we transform our deep Gaussian process into a single global real-valued Gaussian process defined on a compact space in order to apply the results of \cite{van2008rates}. 
 The proofs of the results of Sections \ref{sec:posterior_contraction} and \ref{sec:examples}, and of Appendix \ref{app:inter} are postponed to Appendix \ref{app:proofs}.

\section{Setting and preliminary notation}\label{sec:setting}

\subsection{Gaussian and deep Gaussian priors with constraints}\label{ssec:setting}

Here $\mathbb{N}$ denotes the set of natural numbers (including zero) and
$\mathbb{N}^*$ denotes the set of nonzero natural numbers.
For $k \in \mathbb{N}^*$ and $A \subset \R^k$, let $\mathcal{C}_0(A,\mathbb{R})$ be the set of continuous functions from $A$ to $\mathbb{R}$, endowed with the Borel sigma algebra of the uniform norm $\norme{\cdot{}}_{\infty}$.

Let $H \in \mathbb{N}^*\setminus \{1\}$  and $d \in \mathbb{N}^*$ be fixed. Let also $d_1 = d$ and $d_{H+1} = 1$ by convention.
For $h = 1, \dots , H$, consider a centered multivariate Gaussian process $Z_h = (Z_{h,1} , \dots Z_{h,d_{h+1}}) : \mathbb{R}^{d_h} \to \mathbb{R}^{d_{h+1}}$. Assume that $Z_1 , \dots , Z_H$ are independent, with independent components and with continuous realizations. 
Assume moreover that $Z_2 , \dots , Z_H$ have continuously differentiable realizations.

This paper deals with a deep Gaussian process obtained by composing the previous Gaussian processes $Z_h$, for $h=1,\dots,H$:
\[
Z_H\circ \dots \circ Z_1.
\]
Such a process gives a prior on the continuous functions from $\R^d$ to $\R$. 
An illustration is provided in Figure \ref{fig:gaussien_profond}.

\begin{center}
\begin{figure}[h!]
\begin{tikzpicture}[scale=.9]
\draw[->] (-4,2.5) -- (0,2.5);
\draw[->] (-4,2.5) -- (0,0.5);
\draw[->] (-4,2.5) -- (0,-1.5);

\draw[->] (-4,-1.5) -- (0,2.5);
\draw[->] (-4,-1.5) -- (0,0.5);
\draw[->] (-4,-1.5) -- (0,-1.5);

\draw[->] (1,0.5) -- (5,0.5);
\draw[->] (1,2.5) -- (5,0.5);
\draw[->] (1,-1.5) -- (5,0.5);

\draw[->] (6,0.5) -- (10,0.5);

\draw (-5,2) rectangle (-4,3);
\draw (-5,-2) rectangle (-4,-1);

\draw (0,2) rectangle (1,3);
\draw (0,0) rectangle (1,1);
\draw (0,-2) rectangle (1,-1);

\draw (5,0) rectangle (6,1);

\draw (10,0) rectangle (11,1);

\node at (-4.5,2.5) {$t_1$};
\node at (-4.5,-1.5) {$t_2$};
\node at (-4.5,3.5) {$t=(t_1,t_2)$};

\node at (.5,2.5) {$Z_{1,1}$};
\node at (.5,-1.5) {$Z_{1,3}$};
\node at (.5,0.5) {$Z_{1,2}$};
\node at (.5,3.5) {$Z_1=(Z_{1,1},Z_{1,2},Z_{1,3})$};

\node at (5.5,0.5) {$Z_2$};
\node at (10.5,0.5) {$Z_3$};
\end{tikzpicture}
\caption{Example of a deep Gaussian process from $\R^2$ to $\R$.  Here $H=3$, $d_1=d=2$, $d_2=3$, $d_3=1$, and $d_{H+1}=d_4=1$.}
\label{fig:gaussien_profond}
\end{figure}
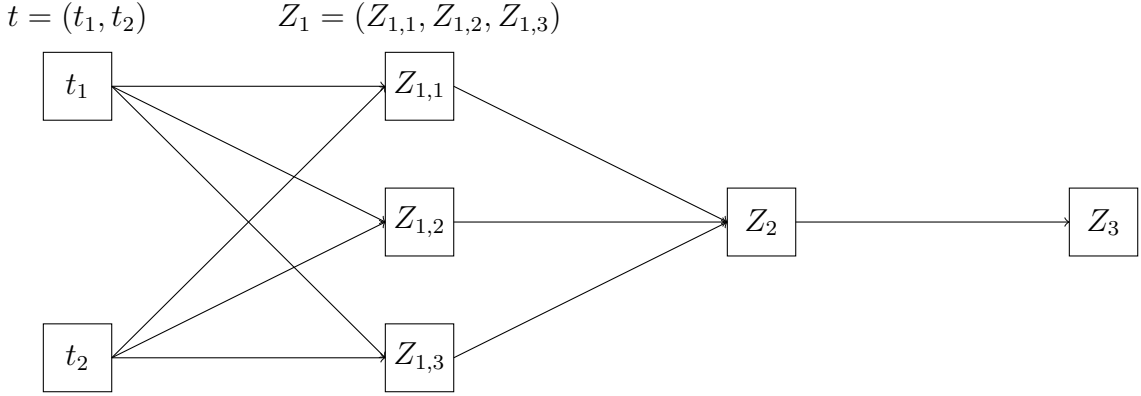
\end{center}

We consider the input space $[-1,1]^d$ for the deep Gaussian process prior $Z_H \circ \dots \circ Z_1$. 
For our proof techniques, we constrain the process $Z_h$, $h=2,\dots,H$, to also have the fixed input domain $[-1,1]^{d_h}$. The reader is referred to Appendix \ref{app:inter} for  more details.  
To do so,  we condition $Z_1 , \dots , Z_{H-1}$ by the following constraints, that we call \emph{value constraints}: for all $h=1 , \dots , H-1$, 
 \begin{equation} \label{eq:first:constraint}
 \norme{ Z_h }_{\infty,[-1,1]^{d_h}} \leqslant 1,
 \end{equation}
where we define for any $k_1,k_2 \in \mathbb{N}^*$, for any $A \subset \mathbb{R}^{k_1}$, and for any $f = (f_1 , \dots , f_{k_2}) \colon A \to \mathbb{R}^{k_2}$, 
 \[
 \norme{f}_{\infty,A} =   \sup_{t\in A} \max_{\ell=1,\dots,k_2} |f_\ell(t)|.
 \]
 
 This set of constraints has strictly positive probability (see Lemma \ref{lemma:proba:small:ball:non-zero}). 
 In addition, for technical reasons again (in particular for the proof of Lemma \ref{lem:norme:infini:composition}), we consider a second set of constraints, that we call \emph{derivative constraints}: 
for $h=2 , \dots , H$, for $i = 1 , \dots ,d_{h+1}$, and for $j = 1 , \dots , d_{h}$, there exists a fixed deterministic constant $K_{h,i,j}$ such that 
 \begin{equation} \label{eq:second:constraint}
 \norme{
  \frac{ \partial   Z_{h,i} }{\partial x_j}  
}_{\infty ,[-1,1]^{d_h}}
 \leqslant K_{h,i,j}. 
 \end{equation}
 This second set of constraints has again strictly positive probability (see Lemma \ref{lemma:proba:small:ball:non-zero} applied to the derivatives of $Z_{h,i}$ for $h=2 , \dots , H$ and for $i = 1 , \dots ,d_{h+1}$). 
As explained, the previous constraints are necessary for our theoretical analysis and may also be useful in practice to exclude very irregular realizations from the Bayesian model.
Observe that the form of the constraints in \eqref{eq:first:constraint} can be extended to more flexible bounds on the values of the components of $Z_1 , \dots , Z_{H-1}$ (see Appendix \ref{appendix:equivalent:constraints}). Outside of Appendix \ref{appendix:equivalent:constraints}, we consider the constraints \eqref{eq:first:constraint} for convenience of exposition. 

Now we index all the univariate components of $Z_1 , \dots , Z_H$, defining the finite set
 \[
 \mathcal{I}
 =
 \{
 (h, i );\,
h \in \{ 1 ,\dots , H\},
 i \in \{1 , \dots , d_{h+1}\}
 \}.
 \]

In the sequel, for $(h,i)\in \mathcal I$, we denote by $Z_{c,h,i}$ a stochastic process from $[-1,1]^{d_h}$ to $\R$ which law is that of the process $Z_{h,i}$ conditioned by \eqref{eq:first:constraint} and/or \eqref{eq:second:constraint}.   More precisely, for any Borel set $B$, for $h=2,\dots, H-1$, and for $i=1,\dots,d_{h+1}$,
\begin{align}\label{eq:Zcond}
\P&\left(Z_{c,h,i} \in B\right)\nonumber\\
&=\P\left(Z_{h,i} \in B \big\vert \norme{ Z_{h,i} }_{\infty,[-1,1]^{d_h}} \leqslant 1,\ \norme{
   \partial   Z_{h,i} / \partial x_j 
}_{\infty ,[-1,1]^{d_h}}
 \leqslant K_{h,i,j},\ j=1,\dots, d_h \right)
\end{align}
and
\begin{align}\label{eq:Zcond1H}
\begin{cases}
\P\left(Z_{c,1,i} \in B\right)=\P\left(Z_{1,i} \in B \big\vert \norme{ Z_{1,i} }_{\infty,[-1,1]^{d_1}} \leqslant 1 \right) \quad \quad \quad  \text{for $ i=1,\dots,d_2$},\\
\P\left(Z_{c,H} \in B\right)=\P\left(Z_{H} \in B \big\vert\norme{
   \partial   Z_{H} / \partial x_j 
}_{\infty ,[-1,1]^{d_H}}
 \leqslant K_{H,1,j},\ j=1,\dots, d_H  \right).
\end{cases}
\end{align}

For $h=1,\dots,H$, we let $Z_{c,h}=(Z_{c,h,1},\dots,Z_{c,h,d_{h+1}})$. Note
 that the independence of the processes $Z_{h,i}$, for $(h,i)\in \mathcal I$, yields that, for any Borel sets $B_1,\dots,B_H$,
\begin{align*}
&\P\left(Z_{c,h} \in B_h,\  h=1,\dots,H \right)\\
=&\P\bigl(Z_{h} \in B_h, \   h=1,\dots,H \big\vert \norme{ Z_{h,i} }_{\infty,[-1,1]^{d_h}} \leqslant 1,  h=1,\dots, H-1,  i=1\,\dots,d_{h+1},\\\
& \hspace{2.5cm}  \norme{
   \partial   Z_{h,i} / \partial x_j 
}_{\infty ,[-1,1]^{d_h}}
 \leqslant K_{h,i,j}\  h=2,\dots, H,\  i=1\,\dots,d_{h+1},\ j=1\,\dots,d_{h}\bigr).
\end{align*}
  
For $h=1,\dots,H$, let $\mathbb B_h$ be the Banach space of the continuous functions from $[-1,1]^{d_h}$ to $\R$ endowed with the uniform norm.
In addition, for any $h=2,\dots,H-1$ and for any $i=1,\dots,d_{h+1}$, we define the subset $\mathbb B_{c,h,i}$ of $\mathbb B_{h}$ composed by the continuously differentiable functions $z \in \mathbb{B}_h$ satisfying \eqref{eq:first:constraint} and \eqref{eq:second:constraint} (with $Z$ replaced by $z$). Similarly, for $i=1,\dots,d_2$, $\mathbb B_{c,1,i}$ stands for the subset of the functions $z \in \mathbb{B}_1$ satisfying \eqref{eq:first:constraint}
only and $\mathbb{B}_{c,H,1}$ stands for the subset of the continuously differentiable functions $z \in \mathbb{B}_H$ satisfying \eqref{eq:second:constraint} only.
Hence \eqref{eq:Zcond} and \eqref{eq:Zcond1H}
simply rewrite, for any Borel set $B\subset \mathbb B_h$ and for any $(h,i)\in \mathcal I$,
\[
\P(Z_{c,h,i} \in B)=\P(Z_{h,i} \in B \vert Z_{h,i} \in \mathbb B_{c,h,i}).
\]
The constrained deep Gaussian process prior is then given by
\[
Z_{c,H}\circ \dots \circ Z_{c,1}.
\]

\subsection{A concentration function for constrained deep Gaussian processes}\label{sec:add_notation}

For a centered Gaussian process $X$ on a space $E$, we define its RKHS $\mathbb H_X$ on $E$ from the covariance function of $X$, on $E \times E$, as in \cite[Section 2.1]{van2008reproducing}. 
Then let  $\norme{ \cdot }_{\mathbb H_X}$ be the Hilbertian norm in this RKHS.
In particular, for $(h,i) \in \mathcal{I}$, we let $k_{h,i}$ be the covariance function of $Z_{h,i}$ on $[-1,1]^{d_h}$ and $\mathbb H_{h,i}$ be the corresponding RKHS (that is thus defined as in \cite[Section 2.1]{van2008reproducing}) with RKHS-norm $\norme{ \cdot }_{\mathbb H_{h,i}}$. 
We also write $\mathcal{C}^1( [-1,1]^\ell , \mathbb{R} )$ for the
Banach space  of continuously differentiable functions $f$ from $[-1,1]^\ell$ to $\R$, equipped  with the norm $\norme{\cdot}_{\infty,1}$ defined by 
\begin{align}\label{def:norne:infini:1}
\norme{f}_{\infty,1} =
\norme{f}_{\infty}  
+ \sum_{j=1}^\ell
 \norme{ \frac{ \partial f }{ \partial x_j} }_{\infty}.
 \end{align}

In \cite{van2008rates}, the rates of contraction for classical Gaussian processes are given in terms of a concentration function involving the RKHS associated to the Gaussian process prior. More precisely, the concentration function for a single Gaussian process $X$ from a compact space $E$ to $\R$, with a continuous function $x_0$ on $E$, and $\varepsilon>0$, is given by
\begin{align}
\phi_{x_0}(\varepsilon)
= &
\underset{
\substack{ h \in \mathbb H_X \\
\norme{ h - x_0 }_{\infty} < \varepsilon 
} 
}{\inf}
\norme{h}_{\mathbb H_X}^2
-
\log
\mathbb{P}
\left(
\norme{ X }_{\infty} 
< \varepsilon
\right)\label{def:w0}
\end{align}
where $\norme{\cdot{}}_\infty$ is the uniform norm  for continuous functions from $E$ to $\R$. This function is interpreted as measuring the concentration around the fixed function $x_0$. 

\medskip

Here, we propose a novel concentration function adapted to the deep context and to the constraints. First, for $h = 2,\dots,H$ and for $i=1,\dots, d_{h+1}$, we consider $z_{0,h,i}$ in the closure of $\mathbb H_{h,i}$ in $(\mathcal{C}^1( [-1,1]^{d_h} , \mathbb{R} ) ,\norme{\cdot}_{\infty,1} )$ and for $i =1 , \dots , d_2$, we consider $z_{0,1,i}$ in the closure of $\mathbb H_{1,i}$ in $( \mathcal{C}^0( [-1,1]^{d_1} , \mathbb{R} ) , \norme{\cdot}_{\infty} )$. Second, we define, for $\varepsilon >0$,
\begin{align}
\Phi_{c,z_0}(\varepsilon)  
\defeq & 
	\sum_{i = 1 }^{d_2}
	\left(
	\frac{3}{2} 	\inf_{
	\substack{
		g \in \mathbb H_{1,i} \\
		\norme{  g - z_{0,1,i}}_{\infty} < \varepsilon
	}
} 
	\norme{  g }_{\mathbb H_{1,i}}^2
- 2 \log 
\mathbb{P}
\left(
\norme{Z_{1,i}}_{\infty} < \varepsilon 
\right)
	\right) 
\notag \\
& +	\sum_{
	\substack{
		(h,i) \in \mathcal{I} \\
		h \geqslant 2
	}
}
\Bigg( 
	\frac{3}{2}
\underset{
	\substack{
		g \in \mathbb H_{h,i} \\
		\norme{ g - z_{0,h,i} }_{\infty} < \frac{\varepsilon}{2} \\
		\norme{ \partial g / \partial x_j - \partial z_{0,h,i} / \partial x_j }_{\infty} < \frac{K_{\min}}{4}, \\
		j = 1 , \dots , d_h
	} 
}{\inf}
\norme{g}_{\mathbb H_{h,i}}^2 \notag
\\
&
- 2 \log \mathbb{P}
\Bigl(
\norme{ 
	Z_{h,i}
}_{\infty} \leqslant  \frac{\varepsilon}{2}
\Bigr)
	- 2 
\sum_{j = 1}^{d_h}
\log \mathbb{P}
\Bigl(
\norme{ 
	\partial Z_{h,i} / \partial x_j
}_{\infty}
\leqslant  \frac{K_{\min}}{4}
\Bigr)
\Bigg),\label{def:Phi_w0}
\end{align}
where 
  \begin{equation} \label{eq:Kmin}
  K_{\min} = \min_{h=2 , \dots, H}
  \min_{i=1 , \dots , d_{h+1}}
  \min_{j=1 ,\dots , d_h}
  K_{h,i,j}.
  \end{equation}

This function is interpreted as evaluating the concentration  of the processes $Z_{h,i}$ around the fixed functions $z_{0,h,i}$ for $(h,i)\in \mathcal I$, taking the constraints into account.

\section{Posterior contraction in density estimation and classification}
\label{sec:posterior_contraction}

\subsection{Density estimation} \label{ssec:density}

Consider the observation of  an i.i.d.\ sample $X_1$, \dots, $X_n$ from a fixed unknown probability density function $p_0\colon [-1,1]^d  \to (0,+\infty)$ (such that $\int_{[-1,1]^d} p_0(t)dt=1$). 
For $z \colon [-1,1]^d \to \R$, let $p_z \colon [-1,1]^d \to (0,+\infty)$ be given by
\begin{align}\label{def:dens}
p_z(t)=\frac{e^{z(t)}}{\int_{[-1,1]^d}e^{z(s)}ds}, \qquad \text{for $t\in [-1,1]^d$}.
\end{align}
Suppose that $p_0$ can be written as $ p_{z_{0,H}\circ \cdots \circ z_{0,1}}$ where $z_{0,h}=(z_{0,h,1},\dots, z_{0,h,d_{h+1}})$ for  $h=1,\dots, H$ and such that $z_{0,h,i}\in \mathbb B_{c,h,i}$ for $(h,i)\in \mathcal I$. 
As above, assume that, for $h =2,\dots,H$ and for $i=1,\dots,d_{h+1}$, $z_{0,h,i}$ is in the closure of $\mathbb H_{h,i}$ in $(\mathcal{C}^1( [-1,1]^{d_h} , \mathbb{R} ) ,\norme{\cdot}_{\infty,1} )$ and for $i =1 , \dots , d_2$, $z_{0,1,i}$ is in the closure of $\mathbb H_{1,i}$ in $( \mathcal{C}^0( [-1,1]^{d} , \mathbb{R} ) , \norme{\cdot}_{\infty} )$.

We consider the prior $p_{Z_{c,H}\circ \cdots \circ Z_{c,1}}$ on $p_0$, where, for any $h=1,\dots, H$, $Z_{c,h}$ has been defined in Section \ref{ssec:setting}. Then the posterior is given by 
\[
\P(  p_{Z_{c,H} \circ \dots \circ Z_{c,1}} \in \cdot{} ~| X_1 ,\dots , X_n ).
\] 
The posterior is a random measure on $( \mathcal{C}^0( [-1,1]^{d} , \mathbb{R})$ that depends on the observations $X_1,\dots,X_n$ and we are interested in its convergence to $p_0$ as $n \to \infty$.

Before stating the main result of this section, we recall the definition of the Hellinger distance $h$: for $f,g : [-1,1]^d \to [0,\infty)$ with $\int_{[-1,1]^d} f(t) dt = \int_{[-1,1]^d} g(t) dt = 1$, 
\begin{align}\label{def:hellinger}
h(f,g)
=
\sqrt{ 
\int_{[-1,1]^d} \left( \sqrt{ f(t)}-\sqrt{ g(t)} \right)^2
dt
}.
\end{align}

The next theorem is an extension of  \cite[Theorem 3.1]{van2008rates} to constrained deep Gaussian priors.

\begin{theorem}\label{th:Th3.1}
Let  $\Phi_{c,z_0}$ be the function defined in \eqref{def:Phi_w0} and recall $K_{\min}$ in \eqref{eq:Kmin}. Assume that, for $h=1,\dots, H-1$ and  for $i=1,\dots,d_{h+1}$, $\norme{z_{0,h,i}}_\infty < 1$ and for $h= 2,\dots,H$, for $i=1,\dots,d_{h+1}$, and for $j=1,\dots,d_h$, $\norme{\partial z_{0,h,i}/\partial x_j}_\infty\leqslant K_{\text{min}}/2$. Then,
for a sequence $(\varepsilon_{c,n})_{n\in \N}$ satisfying 
$\varepsilon_{c,n} > 0$, $\varepsilon_{c,n} \to 0$, $ n \varepsilon_{c,n}^2 \to \infty$, and 
$\Phi_{c,z_0}(\varepsilon_{c,n}) \leqslant n \varepsilon_{c,n}^2$, the posterior distribution satisfies 
\begin{equation} \label{eq:posterior:rate:density}
\mathbb{P} \left(
h(p_{Z_{c,H}\circ \cdots \circ Z_{c,1}},p_0)>M_n\varepsilon_{c,n}\vert X_1,\dots,X_n 
\right)
\underset{n\to\infty}{\to} 0
\end{equation}
in probability, for any sequence $(M_n)$ such that $M_n \to \infty$. 
\end{theorem}

Remark that in Theorem \ref{th:Th3.1}, the conditions $\norme{z_{0,h,i}}_\infty < 1$ and $\norme{\partial z_{0,h,i}/\partial x_j}_\infty\leqslant K_{\text{min}}/2$ are not restrictive since, for a given $p_0$, $z_{0,H}$ can be taken with large enough values and derivatives values, to compensate for the bounded values of $z_{0,1} , \dots , z_{0,H-1}$, and then the $K_{h,i,j}$ in \eqref{eq:Kmin} can be taken large enough. Recall also that Appendix \ref{appendix:equivalent:constraints} shows that the value bound $1$ can be replaced by arbitrary large bounds, up to linear changes of inputs and outputs of the Gaussian processes.

\subsection{Classification}\label{ssec:classif}

Consider the observation of an i.i.d.\ sample $(U_1,V_1) , \dots , (U_n , V_n)$ distributed as $(U,V)$ where $U$ is a random variable in $[-1,1]^d$ and $V$ is a binary random variable ($V \in \{ 0 , 1 \}$) such that $\P( V=1 | U ) = f_0 (U) $, with a fixed unknown function $f_0 : [-1,1]^d \to (0,1)$. Consider a function $\Psi : \mathbb{R} \to (0,1)$ such that $\Psi$ has a bounded derivative $\Psi'$  on $\R$ and such that $\Psi' / (\Psi (1 - \Psi))$ is bounded on $\R$.   For instance, one may use the standard logistic function $\Psi(x)=e^x/(1+e^x)$. 
For $z\colon [-1,1]^d \to \R$, let $f_z \colon [-1,1]^d \to (0,1)$ be given by
\begin{align}\label{def:classif}
	f_z(t)= \Psi(z(t)), \qquad \text{for $t\in [-1,1]^d$}.
\end{align}

Suppose that $f_0$ can be written as $f_{z_{0,H}\circ \cdots \circ z_{0,1}}$ where $z_{0,h}=(z_{0,h,1},\dots, z_{0,h,d_{h+1}})$ for  $h=1,\dots, H$ and such that $z_{0,h,i}\in \mathbb B_{c,h,i}$ for $(h,i)\in \mathcal I$. Further, as above, assume that, for $h = 2,\dots,H$  and for $i=1,\dots,d_{h+1}$, $z_{0,h,i}$ is in the closure of $\mathbb H_{h,i}$ in $(\mathcal{C}^1( [-1,1]^{d_h} , \mathbb{R} ) ,\norme{\cdot}_{\infty,1} )$ and for $i =1 , \dots , d_2$, $z_{0,1,i}$ is in the closure of $\mathbb H_{1,i}$ in $( \mathcal{C}^0( [-1,1]^{d} , \mathbb{R} ) , \norme{\cdot}_{\infty} )$.

We consider the prior $f_{Z_{c,H}\circ \cdots \circ Z_{c,1}}$, where, for $h=1,\dots, H$, $Z_{c,h}$ has been defined in Section \ref{ssec:setting}. Then the posterior is given by
\[
\P(  f_{Z_{c,H} \circ \dots \circ Z_{c,1}} \in \cdot{} ~| 
(U_1,V_1) , \dots , (U_n , V_n)
 ).
\] 
Again, the posterior is a random measure on $( \mathcal{C}^0( [-1,1]^{d} , \mathbb{R})$ that depends on the observations $(U_1,V_1),\dots,(U_n,V_n)$ and we are interested in its convergence to $f_0$ as $n\to \infty$.

We let $ \norme{g}_{2,U} $ be the $L^2$ norm of a function $g : [-1,1]^d \to \R$ with respect to the law of $U$.
The next theorem is an extension of  \cite[Theorem 3.2]{van2008rates} (see also \cite[Theorem 11.22]{ghosal2017fundamentals}) to constrained deep Gaussian priors. 

\begin{theorem}\label{th:Th3.2}
Consider the same setting and the same assumptions as in Theorem \ref{th:Th3.1} for $\Phi_{c,z_0}$, $K_{\min}$, and $(z_{0,h,i})_{(h,i)\in \mathcal I}$. 
Then, for a sequence $(\varepsilon_{c,n})_{n\in \N}$ satisfying 
	$\varepsilon_{c,n} > 0$, $\varepsilon_{c,n} \to 0$, $ n \varepsilon_{c,n}^2 \to \infty$, and 
	$\Phi_{c,z_0}(\varepsilon_{c,n}) \leqslant n \varepsilon_{c,n}^2$, the posterior distribution satisfies 
	\begin{equation} \label{eq:posterior:rate:classification}
		\mathbb{P} \left(
		\norme{f_{Z_{c,H}\circ \cdots \circ Z_{c,1}} - f_0 }_{2,U}
		>M_n\varepsilon_{c,n}\vert (U_1,V_1),\dots,(U_n,V_n)
		\right)
		\underset{n\to\infty}{\to} 0
	\end{equation}
	in probability, for any sequence $(M_n)$ such that $M_n \to \infty$. 
\end{theorem}

\section{Examples of rates of contraction for specific function classes and specific priors}\label{sec:examples}

In this section, we apply Theorems \ref{th:Th3.1} and \ref{th:Th3.2} to classical examples of Gaussian processes for $Z_1, \dots, Z_H$  and to classical examples of function classes for $z_{0,1},\dots,z_{0,H}$. These examples are by no means exhaustive, and Theorems \ref{th:Th3.1} and \ref{th:Th3.2} would enable to obtain contraction rates in other settings as well.

\subsection{Integrated Brownian motion processes in dimension one} \label{ssec:integrated:brownian}

First consider independent Brownian motions as studied in \cite{van2008rates} for standard (non deep) Gaussian processes.  
Here, we let $d_1 = \dots = d_{H+1} = 1$, that is we consider the composition of univariate processes.  
Thus we simply have $Z_h=Z_{h,1}$  for $h=1,\dots,H$ and the constants $K_{h,i,j}$ involved in the constraints \eqref{eq:second:constraint} will simply be denoted by $K_h$. For $x \in \mathbb{R}$, let $\lfloor x \rfloor $ be the largest integer strictly smaller than $x$. 
For $\beta >0$,
let $\mathcal{F}^{\beta}([-1,1],\mathbb{R})$ be the set of functions from $[-1,1]$ to $\mathbb{R}$ that have derivatives up to order $\lfloor \beta \rfloor$ and which derivative of order $\lfloor \beta \rfloor$ is $\beta - \lfloor \beta \rfloor$-H\"older. The space  $\mathcal{F}^{\beta}([-1,1],\mathbb{R})$ is called the H\"older space of order $\beta$.

For a continuous function $g : [0 , \infty) \to \mathbb{R}$, we let $I^0 g = g$ and, for $\ell \in \mathbb{N}^*$, we define by induction $I^\ell g : [0 , \infty) \to \mathbb{R}$ by, for $t \in [0,\infty) $, $(I^\ell g) (t) = \int_0^t (I^{\ell-1} g) (s) ds$. 
We let $N_1 \in \mathbb{N}$ and, for $h =2 , \dots , H$, $N_{h} \in \mathbb{N}^\star$. We also let $Z_h$ be the Gaussian process defined by
\[
Z_h(t) =  (I^{N_{h}} B_{h}) (t+1)
+ \sum_{\ell = 0}^{N_{h}} X_{h,\ell} \frac{ (t+1)^\ell}{\ell!}, \qquad \text{for $t \in (-1,\infty)$},
\]
where $B_{h}$ is a standard Brownian motion on $[0 , \infty)$ and where $X_{h,0}, \dots , X_{h,N_h}$ are independent standard Gaussian variables that are also independent of $B_{h}$. Remark that $Z_h$ is then defined on $(-1,\infty)$, is $N_h$-times differentiable, and has value and derivatives values at $t=-1$ given by $X_{h,0},\dots, X_{h,N_h}$. 

In order to fit with the formalism of Sections \ref{sec:setting} and \ref{sec:posterior_contraction}, we now extend $Z_{1},\dots, Z_{h}$ to the real line. This extension is done simply with a constant or a linear function. This choice is arbitrary and has no impact on the constrained deep Gaussian process priors constructed as in Sections \ref{sec:setting} and \ref{sec:posterior_contraction}. 
  If $N_1 = 0$, we extend the definition of $Z_1(t)$ for $t \in (- \infty,-1)$ by taking the value at $t = -1$. If $N_{h} \geqslant 1$, we extend the definition of $Z_h(t)$ for $t \in (- \infty,-1)$ by a linear function based on the values of  $Z_h$ and its derivative at $t = -1$. After this, $Z_h$ is a continuous Gaussian process on $\mathbb{R}$ (continuously differentiable if $N_{h} \geqslant 1$). 

As in Sections \ref{sec:setting} and \ref{sec:posterior_contraction}, we let $Z_1 , \dots , Z_H$ be independent. The next theorem then provides the rates of contraction given by Theorems \ref{th:Th3.1} and \ref{th:Th3.2}.

\begin{theorem} \label{theorem:rate:holder}
	First consider the setting of Theorem \ref{th:Th3.1}. 
Assume that $p_0$ is strictly positive and belongs to $\mathcal{F}^{\beta}([-1,1],\mathbb{R})$ for a fixed $\beta > 0$.
 Assume that $\beta  \leqslant N_1 + 1/2$ and that $\beta \leqslant N_h$ for $h=2 , \dots , H$. 
Then the fixed constants $K_1 , \dots , K_H$ in \eqref{eq:second:constraint} can be chosen large enough so that the conclusion of Theorem \ref{th:Th3.1} holds with 
\[
\varepsilon_{c,n}
=
C_{\sup} 
n^{-  \beta/(2 N_1  + 2 )}
\]
for some constant $C_{\sup}$ that does not depend on $n$.
Second, consider the setting of Theorem \ref{th:Th3.2}, with $f_0 = f_{g_0}$ where $g_0$ belongs to $\mathcal{F}^{\beta}([-1,1],\mathbb{R})$, with the same conditions on $\beta , N_1 , \dots , N_H$. Then the first conclusion again holds, with the same expression for $\varepsilon_{c,n}$. 
\end{theorem}

From Theorem \ref{theorem:rate:holder} with the condition $\beta \leqslant N_1+1/2$, the rate of contraction is fastest when $N_1 + 1/2=  \beta $ (which is possible when $\beta = k +  1/2 , k \in \mathbb{N},$ and means that the smoothness of $Z_1$ matches that of the fixed unknown function), in which case we recover the classical minimax rate $n^{-\beta/(2\beta+1)}$  (see \cite[Theorem 4.1 and below]{van2008rates}). For a larger $N_1$, the rate deteriorates (oversmooth prior). Note that the processes $Z_2,\dots,Z_H$ are chosen smoother than the fixed unknown function from the condition $N_h\geqslant \beta$ for $h=2,\dots,H$.

\subsection{Riemann-Liouville processes in dimension one}\label{ssec:riemann}

While the smoothness index of the Brownian motion is necessarily of the form integer + 1/2, any smoothness index can be reached with the Riemann-Liouville process as explained in \cite{van2008rates}. 
As in Section \ref{ssec:integrated:brownian}, in the sequel,
we let $d_1 = \dots = d_{H+1} = 1$, we simply have $Z_hZ_{h,1}$, and we simply denote $K_h$ for $K_{h,1,1}$ in \eqref{eq:second:constraint}, $h=1,\dots,H$.  

For $x \in \mathbb{R}$, let $\underline{x}$ be the largest integer smaller or equal to $x$. For $h  = 1 , \dots , H$,  
we set $\alpha_1 >0$ and for $h=2,\dots,H$, $\alpha_{h} > 1$. 
We let $B_h$ be a standard Brownian motion on $[0,\infty)$. We define $R_{h}$ as the Gaussian process on $[ 0 ,\infty)$ defined by $R_{h}(t) = \int_0^t (t-s)^{\alpha_h-1/2} d B_h(s)$
and then the Gaussian process $Z_h$ on $[-1,\infty)$ by 
\[
Z_h(t) = R_{h}(t+1) + \sum_{\ell=0}^{\underline{\alpha}_h+1  } X_{h,\ell} \frac{(t + 1)^\ell}{\ell !}  ,
\]
 for $t \in [-1,\infty)$ and where  $ X_{h,1} , \dots ,  X_{h, \underline{\alpha}_h+1}$ are independent standard Gaussian variables. We then extend $Z_h$ on $(- \infty,1)$
as in Section \ref{ssec:integrated:brownian}.

Again as in Section \ref{ssec:integrated:brownian}, we let $Z_1, \dots , Z_H$ be independent. The next theorem then provides the rates of contraction given by Theorems \ref{th:Th3.1} and \ref{th:Th3.2}.

\begin{theorem} \label{theorem:rate:riemann:liouville}
	First consider the setting of Theorem \ref{th:Th3.1}. 
	Assume that $p_0$ is strictly positive and belongs to $\mathcal{F}^{\alpha_1}([-1,1],\mathbb{R})$. Assume that for $h  = 2 , \dots , H$, $\alpha_ h \geqslant \alpha_1$.	
Then the fixed constants $K_1 , \dots , K_H$ in \eqref{eq:second:constraint} can be chosen large enough so that the conclusion of Theorem \ref{th:Th3.1} holds with 
	\[
	\varepsilon_{c,n}
	=
	C_{\sup} 
	n^{-  \alpha_1/(2 \alpha_1  + 1)}
	\]
	for some constant $C_{\sup}$ that does not depend on $n$.
	Second, consider the setting of Theorem \ref{th:Th3.2}, with $f_0 = f_{g_0}$ where $g_0$ belongs to $\mathcal{F}^{\alpha_1}([-1,1],\mathbb{R})$, with the same conditions on $ \alpha_1 , \dots , \alpha_H$. Then the first conclusion again holds, with the same expression for $\varepsilon_{c,n}$. 
\end{theorem}

Note that the smoothness of $Z_1$ exactly matches that of the fixed unknown function and the processes $Z_2,\dots,Z_H$ have a smoothness that is larger than or equal to that of the fixed unknown function.
This enables to recover the standard minimax rate $n^{-\alpha_1/(2\alpha_1+1)}$  (see \cite[Theorem 4.3 and below]{van2008rates}).

\subsection{Mat\'ern processes in general dimension }\label{ssec:matern}

Here we consider Gaussian processes with Mat\'ern covariance functions, which allows for  both arbitrary input dimension and arbitrary smoothness. These processes were studied in \cite[Section 3.1]{van2011information} for standard (non deep) Gaussian processes.

Let us first extend the definition of $\mathcal{F}^{\beta}([-1,1],\mathbb{R})$ to larger dimensions.
For  $\beta >0$ and $\ell \in \mathbb{N}^*$, we
let  $\mathcal{F}^{\beta}([-1,1]^\ell,\mathbb{R})$ be the set of functions from $[-1,1]^\ell$ to $\mathbb{R}$ for which all the partial derivatives of order $(a_1,\dots,a_\ell)$ exist for all $a_1 , \dots , a_\ell \in \N$ with $a_1 + \dots + a_\ell \leqslant \lfloor \beta \rfloor$ and which partial derivatives of order $\lfloor \beta \rfloor$ are $\beta - \lfloor \beta \rfloor$-H\"older.

For $(h,i) \in \mathcal{I}$, we let $Z_{h,i}$ have Mat\'ern covariance function, that is, for $u,v \in \R^{d_h}$, 
\[
\mathrm{Cov}
\left(
Z_{h,i}(u)
,
Z_{h,i}(v)
\right)
=
\int_{\R^d} e^{i\lambda^\top(u-v)}m_{h,i}(\lambda)d\lambda,
\]
where
\begin{equation} \label{eq:Mhi}
m_{h,i}(\lambda)=\frac{1}{(1+\norme{\lambda}^2)^{\alpha_{h,i}+d/2}}
\end{equation}
is called the spectral density and $\alpha_{h,i} >0$ is called the smoothness parameter. A Gaussian process on $[-1,1]^\ell$ ($\ell \in \N^\star$) with Mat\'ern covariance function  with $\alpha_{h,i} =  \alpha$ in \eqref{eq:Mhi}  has sample paths in $\mathcal{F}^\beta ([-1,1]^{\ell},\R)$ for any $\beta<\alpha$. 

Then, as in \cite{van2011information}, we define the Sobolev space $\mathcal{H}^\beta([-1,1]^\ell,\R)$ as the set of functions $f$ from $[-1,1]^\ell$ to $\R$ that are restrictions of functions $f$ from
$\R^\ell \to \R$ with Fourier transform $\hat f(\lambda)=(2\pi)^{-\ell}\int e^{i\lambda^\top t} f(t)dt$ such that, defining
\[
\norme{f}^2_{ \beta } \defeq \int_{\mathbb{R}^d} (1+\norme{\lambda}^2)^\beta \abs{\hat f^2(\lambda)} d\lambda < \infty.
\]

The next theorem provides the rates of contraction given by Theorems \ref{th:Th3.1} and \ref{th:Th3.2}.

\begin{theorem} \label{theorem:rate:matern}
	First consider the setting of Theorem \ref{th:Th3.1}. 
Assume that $p_0$ is strictly positive and belongs to $\mathcal{F}^{\beta}([-1,1]^d,\mathbb{R})\cap \mathcal{H}^{\beta}([-1,1]^d,\R)$ for some $\beta>0$.
Assume that, for $(h,i) \in \mathcal{I}$, $\beta \leqslant \alpha_{h,i}$. 
Then the fixed constants $K_{h,i,j}$ for $h=2,\dots, H$, for $i=1,\dots,d_{h+1}$, and for $j=1,\dots,d_h$ in \eqref{eq:second:constraint} can be chosen large enough such that the conclusion of Theorem \ref{th:Th3.1} holds with 
\[
\varepsilon_{c,n}
=
C_{\sup} 
n^{-  \beta /
 (2 \alpha_{1,\min}  + d)},
\]
where $\alpha_{1,\min} = \min(\alpha_{1,1} , \dots , \alpha_{1,d_2})$ and
$C_{\sup}$ is some constant that does not depend on $n$.

Second, consider the setting of Theorem \ref{th:Th3.2}, with $f_0 = f_{g_0}$ where $g_0$ belongs to $\mathcal{F}^{\beta}([-1,1]^d,\mathbb{R})\cap \mathcal{H}^{\beta}([-1,1]^d,\R)$, with the same conditions on $\beta , K_{h,i,j}, \alpha_{h,i}$ for $h=1,\dots, H$, for $i=1,\dots,d_{h+1}$, and for $j=1,\dots,d_h$. Then the first conclusion again holds, with the same expression for $\varepsilon_{c,n}$. 
\end{theorem}

Above, the Gaussian process priors are taken smoother than the fixed unknown function. With the appropriate smoothness $\alpha_{1,\min}  = \beta$ for the first layer, we recover the classic optimal rate $n^{- \beta / (2 \beta + d)}$ as pointed out in \cite[Theorem 5]{van2011information}. 

\section{Conclusion}

In this paper, we have provided rates of contraction for the posterior distribution of a deep Gaussian process prior, with constraints on the values and the derivatives. These results are the first to address density estimation and classification, to the best of our knowledge. Recently, \cite{finocchio2021posterior} addressed contraction rates for deep Gaussian processes in regression. 
These general rates (Theorems \ref{th:Th3.1} and \ref{th:Th3.2})  measure in terms of a new concentration function, that takes the constraints into account.  The proofs of these general rates are based on constructing a single global Gaussian process prior in order to exploit existing results for standard Gaussian processes \cite{van2008rates}. Some of the techniques for taking the constraints into account could be considered of independent interest (in particular the proofs in Section \ref{ssec:proof:analyzing:phi}). 

We show in three examples how the general rates enable to recover existing minimax rates for standard function classes and Gaussian processes. These examples are not exhaustive. In future work, it would be beneficial to exploit Theorems \ref{th:Th3.1} and \ref{th:Th3.2} in other settings, with more involved function classes based on composition structures, in the aim of obtaining contraction rates that are out of reach for a single Gaussian  process prior, thus further illustrating the flexibility benefit brought by deep Gaussian processes. Finally, adaptivity to function smoothness is an important topic for future work as well. 

\appendix

\section{More flexibility on the constraints by linear transformations of inputs and outputs} \label{appendix:equivalent:constraints}

Here, we consider that $Z_1,\dots,Z_H$ satisfy constraints of the form \eqref{eq:second:constraint} but the constraints $\norme{ Z_h }_{\infty,[-1,1]^{d_h}} \leqslant 1$ are replaced by more general constraints: for $h=1,\dots, H-1$ and for $i=1,\dots,d_{h+1}$, we assume
\begin{align}\label{eq:first:constraint:general}
\norme{ Z_{h,i} }_{\infty , \prod_{\ell=1,\dots,d_h} [-L_{h,\ell} , L_{h,\ell}]} \leqslant L_{h+1,i}
\end{align}
where the constants $L_{h+1,i}>0$ (for $h=1,\dots,H-1$ and for $i=1,\dots,d_{h+1}$) are arbitrary and $L_{1,1} = \dots = L_{1,d_1}=1$.

We show that we can construct processes $Y_1,\dots,Y_H$ that provide the same composition
\[
Y_H\circ \dots \circ Y_1 = Z_H\circ \dots \circ Z_1,
\]
that have constraints of the form \eqref{eq:first:constraint} and \eqref{eq:second:constraint}, and that are obtained from linear transformations of inputs and outputs. 
Indeed, define the Gaussian processes $Y_1 , \dots , Y_H$ as follows. For $h=1,\dots,H-1$, for $t_1 , \dots , t_{d_h} \in \mathbb{R}$, and for $i = 1 , \dots , d_{h+1}$, let
\begin{equation} \label{eq:fromY:to:Z:one}
	Y_{h,i}(  t_1 , \dots , t_{d_h} )
	=
	\frac{1}{L_{h+1,i} }
	Z_{h,i} \left(
	L_{h,1}  t_1 ,
	\dots,
	L_{h,d_h}  t_{d_h}
	\right).
\end{equation}
Recall that $d_{H+1}=1$ and let
\begin{equation} \label{eq:fromY:to:Z:two}
	Y_{H}(t_1,\dots t_{d_H})
	=
	Z_H \left(
	L_{H,1}  t_1,
	\dots,
	L_{H,d_H}  t_{d_H}  
	\right), \qquad \text{for $t_1,\dots,t_{d_H} \in \mathbb{R}$}.
\end{equation}

The next lemma then shows  the composition equality.

\begin{lemma} \label{lemma:invariance:one}
	For any $t_1 , \dots , t_d \in \mathbb{R}$, we have
	\[
	Y_H \circ \dots \circ Y_1 (t_1 , \dots , t_d) 
	=
	Z_H \circ \dots \circ Z_1 (t_1 , \dots , t_d) 
	\]
recalling the convention $d_1=d$.
\end{lemma}

\begin{proof}[Proof of Lemma \ref{lemma:invariance:one}]
First let us prove by induction that, for $h = 1 , \dots H-1$, for $i =1 , \dots , d_{h+1}$, and for $t_1 , \dots , t_d \in \mathbb{R}$, 
	\begin{equation} \label{eq:invariance:by:induction}
		Y_{h,i} \circ Y_{h-1} \circ \dots \circ Y_1 (t_1 , \dots , t_d) 
		=
		\frac{1}{ L_{h+1,i} }
		Z_{h,i} \circ Z_{h-1} \circ \dots \circ Z_1 (t_1 , \dots , t_d).
	\end{equation}
	For $h=1$, this is true from \eqref{eq:fromY:to:Z:one} with $L_{1,1} = \dots = L_{1,d_1} = 1$ and $d_1=d$. Assume that this is true for some $h \in \{  1 , \dots H-2 \}$. Thus by \eqref{eq:invariance:by:induction} and then by  \eqref{eq:fromY:to:Z:one}, for $i =1 , \dots , d_{h+1}$,
	\begin{small}\begin{align*}
		&	Y_{h+1,i} \circ \dots \circ Y_1 (t_1 , \dots , t_d)
		\\
		&		=
		Y_{h+1,i} 
		\left( 
		Y_{h,1 } \circ Y_{h-1} \circ \dots \circ Y_1 (t_1 , \dots , t_d),
		\dots,
		Y_{h,d_{h+1} } \circ Y_{h-1} \circ \dots \circ Y_1 (t_1 , \dots , t_d)
		\right)
		\\
		& = 
		Y_{h+1,i} 
		\left( 
		\frac{1}{L_{h+1,1}}		 Z_{h,1 } \circ Z_{h-1} \circ \dots \circ Z_1 (t_1 , \dots , t_d),
		\dots,
		\frac{1}{L_{h+1,d_{h+1}}}		 Z_{h,d_{h+1} } \circ Z_{h-1} \circ \dots \circ Z_1 (t_1 , \dots , t_d)
		\right)
		\\
		& =
		\frac{1}{L_{h+2},i}
		Z_{h+1,i} 
		\left( 
		Z_{h,1 } \circ Z_{h-1} \circ  \dots \circ Z_1 (t_1 , \dots , t_d),
		\dots,
		Z_{h,d_{h+1} } \circ Z_{h-1} \circ  \dots \circ Z_1 (t_1 , \dots , t_d)
		\right)
		\\
		& = 
		\frac{1}{L_{h+2},i}
		Z_{h+1,i} \circ Z_{h} \circ \dots \circ Z_1 (t_1 , \dots , t_d).
	\end{align*}
	\end{small}
This concludes the proof of \eqref{eq:invariance:by:induction} by induction. 	
Finally, from \eqref{eq:invariance:by:induction} and  \eqref{eq:fromY:to:Z:two}, we obtain, for $t_1 , \dots , t_d \in \mathbb{R}$,
	\begin{small}
	\begin{align*}
		& Y_H \circ \dots \circ Y_1 (t_1 , \dots , t_d)  \\
		&		= 
		Y_H
		\left(
		Y_{H-1,1} \circ Y_{H-2} \circ \dots \circ Y_1(t_1 , \dots , t_d)
		,
		\dots , 
		Y_{H-1,d_H} \circ Y_{H-2} \circ \dots \circ Y_1(t_1 , \dots , t_d)
		\right)
		\\
		& = 
		Y_H
		\left(
		\frac{1}{L_{H,1}}
		Z_{H-1,1} \circ Z_{H-2} \circ \dots \circ Z_1(t_1 , \dots , t_d)
		,
		\dots , 
		\frac{1}{L_{H,d_H}}
		Z_{H-1,d_H} \circ Z_{H-2} \circ \dots \circ Z_1(t_1 , \dots , t_d)
		\right)
		\\
		& = 
		Z_H
		\left(
		Z_{H-1,1} \circ Z_{H-2} \circ \dots \circ Z_1(t_1 , \dots , t_d)
		,
		\dots , 
		Z_{H-1,d_H} \circ Z_{H-2} \circ \dots \circ Z_1(t_1 , \dots , t_d)
		\right)
		\\
		& = 
		Z_{H} \circ Z_{H-1} \circ \dots \circ Z_1( t_1 , \dots , t_d ).\qedhere
	\end{align*}
	\end{small}
\end{proof}

The next lemma shows that $Y_1,\dots, Y_H$ satisfy constraints of the form \eqref{eq:first:constraint}  when $Z_1,\dots, Z_H$ satisfy the more general ones given in \eqref{eq:first:constraint:general}.

\begin{lemma} \label{lemma:invariance:two}
	The two following assertions are equivalent.
	\begin{enumerate}
		\item For $h = 1 , \dots,H-1$ and for $i = 1 ,\dots , d_{h+1}$,
		\[
		\norme{ Z_{h,i} }_{\infty , \prod_{\ell=1,\dots,d_h} [-L_{h,\ell} , L_{h,\ell}]} \leqslant L_{h+1,i}.
		\]
		\item For $h = 1 , \dots,H-1$ and for $t \in [-1 , 1]^{d_h} $, $Y_h(t) \in [-1,1]^{d_{h+1}}$.
	\end{enumerate}
\end{lemma}

\begin{proof}[Proof of Lemma \ref{lemma:invariance:two}]
	The second assertion can be written as:  for $h = 1 , \dots,H-1$ and $i = 1 ,\dots , d_{h+1}$,
	\[
	\norme{ Y_{h,i} }_{\infty ,  [-1,1]^{d_h}} \leqslant 1.
	\]
	Hence, to prove the lemma, it is sufficient to show that, for $h = 1 , \dots,H-1$ and $i = 1 ,\dots , d_{h+1}$,
	\[
	\norme{ Z_{h,i} }_{\infty , \prod_{\ell=1,\dots,d_h} [-L_{h,\ell} , L_{h,\ell}]} \leqslant L_{h+1,i} 
	\Leftrightarrow
	\norme{ Y_{h,i} }_{\infty ,  [-1,1]^{d_h}} \leqslant 1.
	\]
	The latter equivalence follows from \eqref{eq:fromY:to:Z:one}. 
\end{proof}

Finally, the next lemma shows that $Y_1,\dots,Y_H$ and $Z_1,\dots,Z_H$
equivalently satisfy constraints of the form \eqref{eq:second:constraint}. 

\begin{lemma} \label{lemma:invariance:three}

For $h = 2 , \dots,H$, for $i = 1 , \dots , d_{h+1}$, and for $j = 1 ,\dots , d_h$,
		\[
		\norme{
		\frac{ \partial Z_{h,i} }{\partial x_j}   
		}_{\infty , \prod_{\ell=1,\dots,d_h} [-L_{h,\ell} , L_{h,\ell} ]}
		\leqslant 
	K_{h,i,j}
		\Longleftrightarrow		
				\norme{
		\frac{ \partial Y_{h,i}  }{\partial x_j}   
		}_{\infty , [-1,1]^{d_h}}
		\leqslant   \frac{L_{h,j} }{L_{+1,i}} K_{h,i,j}. 
		\]

\end{lemma}

	\begin{proof}[Proof of Lemma \ref{lemma:invariance:three}]
	Let $L_{H+1,1} = 1$ by convention.
	From  \eqref{eq:fromY:to:Z:one} and \eqref{eq:fromY:to:Z:two}, we have, for $h = 2 , \dots,H$, for $i=1,\dots, d_{h+1}$, $j = 1 ,\dots , d_h$, and for $t_1 , \dots , t_{d_h} \in \mathbb{R}$,
	\begin{align*}
		\left(
		\frac{\partial}{ \partial x_j }	
		Y_{h,i}
		\right)
		(t_1 , \dots , t_{d_h})
		& 	=
		\frac{\partial}{ \partial t_j }	 
		\left(
		\frac{1}{L_{h+1,i}} 
		Z_{h,i} 
		(L_{h,1} t_1 , \dots , L_{h,d_h} t_{d_h})
		\right)
		\\
		&	=
		\frac{ L_{h,j} }{L_{h+1,i}}
		\left( 
		\frac{\partial}{\partial x_j }
		Z_{h,i}
		\right)
		\left(
		L_{h,1} t_1 , \dots , L_{h,d_h} t_{d_h}
		\right).
	\end{align*}
	Hence we have
	\[
	\norme{
	\frac{ \partial Z_{h,i} }{\partial x_j}  
	}_{\infty , \prod_{\ell=1,\dots,d_h} [-L_{h,\ell} , L_{h,\ell} ]}
	\leqslant  K_{h,i,j}
	~ ~
	\Longleftrightarrow
	~ ~	\norme{
	\frac{ \partial Y_{h,i} }{\partial x_j}  
	}_{\infty , [-1,1]^{d_h}}
	\leqslant \frac{ L_{h,j} }{L_{h+1,i}} K_{h,i,j}. 
	\]
	which concludes the proof.
\end{proof}

\section{Preliminary notation and intermediate results}\label{app:inter}

The proofs will exploit the results established in \cite{van2008rates} for a single Gaussian process. Thus we will consider the setting of a single  Gaussian process defined on a compact space and valued in $\R$. The following section is dedicated to the construction of this process.

\subsection{Introduction of a single global Gaussian prior}\label{app:global:prior}

Let $d_{\max} = \max(d_1 , \dots , d_H)$ and let $\mathcal X$ be the compact subspace $[-1,1]^{d_{\max}}\times \mathcal I$ with the distance $d$ defined by 
\[
d((t,h,i),(t',h',i'))=\norme{(t,\sigma(h,i))-(t',\sigma(h',i'))}, \qquad \text{for $(t,h,i)$ and $(t',h',i')$ in $\mathcal X$}
\]
where the norm is the Euclidean norm in dimension $d_{\max}+1$ and $\sigma$ is any fixed bijection from $\mathcal{I}$ to $\{1,\dots,| \mathcal{I} |\}$ with $|  \mathcal{I} |$ the cardinality of $ \mathcal{I} $. 
We introduce the centered Gaussian process $W$, defined by
\begin{align}\label{def:bivariate:proc}
W(t , h,i)
= Z_{h,i}(T_h(t)), \qquad \text{for $(t,h,i) \in \mathcal X$,}
\end{align} 
where $T_h(t)$ is the vector of length $d_h$ obtained from the $d_h$ first coordinates of $t$. Then $W$ has continuous trajectories from $\mathcal X$ to $\R$.

\medskip

Let $\mathbb B$ be the Banach space of the continuous functions from $\mathcal{X}$ to $\R$ endowed with the uniform norm 
given by 
\begin{align*}
\norme{w}_{\infty,\mathcal X}\defeq  \vee_{(h,i)\in \mathcal I}  \norme{w(\cdot{},h,i)}_{\infty,[-1,1]^{d_{\max}}} 
\end{align*}
for any $w\in \mathbb B$. Here $t\vee s$ stands for $\max(t,s)$ between $t , s \in \R$.  From now on, we simply denote $\norme{\cdot{}}_{\infty, \mathcal{X} }$ by $\norme{\cdot{}}_{\infty}$. 
For $w \in \mathbb{B}$ and $(h ,i) \in \mathcal{I}$, we define $P_{h,i}(w)$ as the real-valued function on $[-1,1]^{d_h}$, defined by
\[
P_{h,i}(w)(t)
=
w( (t,0) , h,i), \qquad \text{for $t \in [-1,1]^{d_h}$,}
\]
where the vector $(t,0)$ has dimension $d_{\max}$. Consequently,  for $(h ,i ) \in \mathcal{I}$ and for $t \in [ -1,1 ]^{d_h}  $,
\[
P_{h,i}( W ) (t)
= 
Z_{h,i}(t). 
\]

\medskip

Now we define the subset $\mathbb B_c$ of $\mathbb B$ composed by the functions $w \in \mathbb{B}$ such that the three following conditions (corresponding to \eqref{eq:first:constraint} and \eqref{eq:second:constraint}) hold.
\begin{itemize}
	\item For $h = 1 , \dots , H-1$, $i = 1 ,\dots , d_{h+1}$, $\norme{ P_{h,i}(w) }_{\infty, [-1,1]^{d_h}} \leqslant 1$.
	\item For $h=2 , \dots , H$, $i = 1 , \dots , d_{h+1}$, $P_{h,i}(w)$ is continuously differentiable on $[-1,1]^{d_h}$.  
	\item For $h=2 , \dots , H$, $i = 1 , \dots , d_{h+1}$, $j = 1 , \dots , d_h$, $\norme{ \partial  P_{h,i}(w) / \partial x_j }_{\infty, [-1,1]^{d_h}} \leqslant K_{h,i,j}$.
\end{itemize}

Then we consider a process $W_c$ valued in $\mathbb B$ with distribution defined by
\begin{align}\label{def:Wc}
\P(W_c \in B)=\P(W\in B\vert \ W\in \mathbb B_c),
\end{align}
for any Borel set $B\subset \mathbb B$. Hence $W_c$ corresponds to the Gaussian process $W$ conditioned by the constraints \eqref{eq:first:constraint} and/or \eqref{eq:second:constraint}.

For any function $\psi \in \mathbb B_c$, let $C_{\psi}$ be the function from $[-1,1]^d$ to $\R$ defined by the following.  For $h = 1, \dots, H$, we let the function $C_{\psi,h} = ( P_{h,1}( \psi ) , \dots , P_{h,d_{h+1}}(\psi) )$ be defined on $[-1,1]^{d_h}$ and valued in $[-1,1]^{d_{h+1}}$. Then we let, for $t \in [-1,1]^d$,
\begin{align}\label{def:Cpsi}
C_{\psi}(t)= 
C_{\psi,H} \circ \dots \circ C_{\psi,1}(t).
\end{align}
Remark that, for $t\in [-1,1]^d$, when $W\in \mathbb B_c$, 
\begin{align}\label{def:C_W}
C_{W} (t) = 
Z_H \circ \dots \circ Z_1 (t).
\end{align}

Hence, considering the conditioned version $W_c$ of $W$ (as defined in \eqref{def:Wc}),  $C_{W_c}$ has the distribution of the deep Gaussian process $Z_H \circ \dots \circ Z_1$ where the univariate Gaussian processes $Z_{h,i}$ are conditioned to  \eqref{eq:first:constraint} and/or \eqref{eq:second:constraint}. 

\medskip

Finally, we let $k$ be the covariance function of $W$ on $\mathcal{X} \times \mathcal{X}$ and $\mathbb H$ be the RKHS of $W$ (that is defined as in  \cite[Section 2.1]{van2008reproducing}) with RKHS-norm $\norme{\cdot{}}_{\mathbb H}$.

\subsection{The global concentration function is upper bounded} 

For $w_0$ in the closure of $\mathbb{H}$ in $\mathbb{B}$ (with respect to $\|\cdot{}\|_{\infty}$), we define, for $\varepsilon >0$,
\begin{align} 
\phi_{c,w_0}(\varepsilon)
= &
\underset{
\substack{ h \in \mathbb H \\
\norme{ h - w_0 }_{\infty} < \varepsilon 
} 
}{\inf}
\norme{h}_{\mathbb H}^2
-
\log
\mathbb{P}
\left(
\norme{ W_c }_{\infty} 
< \varepsilon
\right) -\log \P(\norme{W -w_0}_{\infty}<2\varepsilon,\ W\in \mathbb B_c).\label{def:w0_c}
\end{align}
The function $\phi_{c,w_0}$ is interpreted as the concentration function of $W$ around $w_0$ and  is an extension of the function $\phi_{x_0}$ already defined in \eqref{def:w0} and in \cite[Equation (1.2)]{van2008rates}, that takes the bound constraints in $W_c$ into account. The next lemma shows that the presence of the constraints increases the concentration function.

\begin{lemma}\label{lem:ineg_phi_phi_c}
Let $\phi_{w_0}(\varepsilon)$ be as in \eqref{def:w0} with $x_0=w_0$, $X=W$ and $E=\mathcal X$. One has $\phi_{w_0}(\varepsilon)\leqslant \phi_{c,w_0}(\varepsilon)$ for any $\varepsilon >0$. 
\end{lemma}

\begin{proof}[Proof of Lemma \ref{lem:ineg_phi_phi_c}]
For any $\varepsilon>0$, let us consider the  difference between the two concentration functions
\begin{align*}
\phi_{c,w_0}(\varepsilon)-\phi_{w_0}(\varepsilon)&=-
\log
\mathbb{P}
\left(
\norme{ W_c }_{\infty} 
< \varepsilon
\right)-\log \P(\norme{W -w_0}_{\infty}<2\varepsilon,\ W\in \mathbb B_c)\\
&\quad  +
\log
\mathbb{P}
\left(
\norme{ W }_{\infty} 
< \varepsilon
\right).
\end{align*}
Since $\mathbb{P}
\left(
\norme{ W_c }_{\infty} 
< \varepsilon
\right)=\mathbb{P}
\left(
\norme{ W }_{\infty} 
< \varepsilon,\ W\in \mathbb B_c 
\right)/
 \P(W\in \mathbb B_c)$, 
\begin{align*}
&\P\left(
\norme{ W }_{\infty} 
< \varepsilon
\right)
\geqslant 
\mathbb{P}
\left(
\norme{ W }_{\infty} 
< \varepsilon,\ W\in \mathbb B_c 
\right)  \text{ and }
\P(W\in \mathbb B_c) \geqslant  \P(\norme{W -w_0}_{\infty}<2\varepsilon,\ W\in \mathbb B_c),
\end{align*}
the latter difference is non-negative.
\end{proof}

The next theorem shows that the global $\phi_{c,w_0}$ defined in \eqref{def:w0_c} is upper bounded by $\Phi_{c,w_0}$ defined in \eqref{def:Phi_w0}.

\begin{theorem} \label{thm:analyzing:phi}
Let $w_0$ in $\mathbb{B}$.
For $(h,i) \in \mathcal{I}$, assume that $w_0(  \cdot ,h , i ) = P_{h,i}(w_0) (T_h(\cdot))$.  
For $h = 2,\dots,H$ and for $i=1,\dots,d_{h+1}$,	assume that $P_{h,i}(w_0)$ is in the closure of $\mathbb H_{h,i}$ in $(\mathcal{C}^1( [-1,1]^{d_h} , \mathbb{R} ) ,\norme{\cdot}_{\infty,1} )$. Recall that $\norme{\cdot}_{\infty,1}$ is defined in \eqref{def:norne:infini:1}. 
For $i =1 , \dots , d_2$, assume that $P_{1,i}(w_0)$ is in the closure of $\mathbb H_{1,i}$ in $( \mathcal{C}^0( [-1,1]^{d_1} , \mathbb{R} ) , \norme{\cdot}_{\infty} )$.
Then $w_0$ is in the closure of $\mathbb H$ in $(\mathbb{B} , \norme{ \cdot }_{\infty})$. 

Now consider $ \varepsilon \in (0,1]$ and assume that, for $h=1,\dots,H-1$, for $i=1 , \dots , d_{h+1}$,
\begin{equation} \label{eq:analyzing:phi:cond:un}
	\norme{  P_{h,i}(w_0)  }_{\infty}
	+ 2 \varepsilon
	\leqslant 1
\end{equation}
and, for  $h = 2,\dots,H$, for $i=1,\dots,d_{h+1}$, and for $j=1,\dots,d_h$,
\begin{equation} \label{eq:analyzing:phi:cond:deux}
\norme{  
	\frac{
\partial			P_{h,i}(w_0) 
	}{
\partial x_j
}  }_{\infty}
\leqslant \frac{K_{\min}}{2}
\end{equation}
where $K_{\min}$ is defined in \eqref{eq:Kmin}.
Then we have
\begin{align*}
\phi_{c,w_0}(\varepsilon)  
\leqslant & \Phi_{c,w_0}(\varepsilon)
\end{align*}
where $\Phi_{c,w_0}$ is as in \eqref{def:Phi_w0} with $z_{0,h,i}=P_{h,i}(w_0)$ for $(h,i)\in \mathcal I$.
\end{theorem}


\subsection{Contraction rates with the global concentration function}

The next theorem is an adaptation of  \cite[Theorem 2.1]{van2008rates}. For a metric space $(E,d)$, any subset $A \subset E$, and any $\varepsilon >0$, we let $N( \varepsilon,A,d )$ be the minimum number of balls of radius $\varepsilon$ needed to cover $A$.

\begin{theorem} \label{th:Th2.1}
Consider any sequence $(\varepsilon_{c,n})_{n\in \N}$, satisfying $n \varepsilon_{c,n}^2 \to \infty$ and $\phi_{c,w_0} (\varepsilon_{c,n}) \leqslant n \varepsilon_{c,n}^2$ for $\phi_{c,w_0}$ defined in \eqref{def:w0_c}, and any $C>1$ with $e^{-Cn\varepsilon_{c,n}^2}<1/2$. Then there exists an integer $n_C\geqslant 0$ and a sequence of measurable sets $(B_{c,n})_{n\in \N}$ with $B_{c,n}\subset \mathbb B_c$ such that, for $n\geqslant n_C$,  
\begin{align}
\log N( 6 \sqrt{C} \varepsilon_{c,n} ,B_{c,n},\norme{\cdot{}}_{\infty}) 
&\leqslant 
n  (12 \sqrt{C} \varepsilon_{c,n})^2,\label{eq:Th2.1:1}\\
\P(W_c\notin B_{c,n})&\leqslant e^{-Cn\varepsilon_{c,n}^2},\label{eq:Th2.1:2}\\
\P(\norme{W_c-w_0}_{\infty}<2\varepsilon_{c,n})&\geqslant e^{-n\varepsilon_{c,n}^2}.\label{eq:Th2.1:3}
\end{align}
\end{theorem}

Theorem \ref{th:Th2.1} enables to obtain posterior contraction rates at speed $\varepsilon_{c,n}$ satisfying $\phi_{c,w_0} (\varepsilon_{c,n}) \leqslant n \varepsilon_{c,n}^2$ similarly as done in \cite{van2008rates}. This directly allows to prove Theorems \ref{th:Th3.1} and \ref{th:Th3.2}, that is contraction rates at speed $\varepsilon_{c,n}$ satisfying 
$\Phi_{c,z_0} (\varepsilon_{c,n}) \leqslant n \varepsilon_{c,n}^2$
from the bound $\phi_{c,w_0}\leqslant \Phi_{c,z_0}$ of Theorem \ref{thm:analyzing:phi}.

\section{Proofs}\label{app:proofs}

In the rest of the appendix, we write $C_{\sup}$ for a finite constant which value is allowed to change between occurrences.

\subsection{Proof of Theorem \ref{thm:analyzing:phi}} \label{ssec:proof:analyzing:phi}

Before proving Theorem \ref{thm:analyzing:phi}, let us establish several useful lemmas.
We do not exclude that results similar to Lemma \ref{lem:RKHS:X} could be known by the experts, but we have not found any reference. We prove Lemma \ref{lem:RKHS:X} after its statement for self-sufficiency.

\begin{lemma} \label{lem:RKHS:X}
The RKHS $\mathbb H$ of the covariance function $k$ is equal to the set of functions $g : \mathcal{X} \to \mathbb{R}$ such that, for $(h,i) \in \mathcal{I}$,
\begin{equation} \label{eq:structure:RKHS:X}
 g( \cdot ,h ,i ) =  P_{h,i}(g)(T_h(\cdot)) 
 ~  ~ ~
 \text{and}
 ~ ~ ~
  P_{h,i} (g) \in \mathbb H_{h,i}.
\end{equation}
Furthermore, for $g \in \mathbb H$, one has
\begin{equation} \label{eq:norm:equality:bivariate}
\norme{ g }^2_{\mathbb H}
=
\sum_{(h,i) \in \mathcal{I}}
\norme{ P_{h,i} (g) }^2_{\mathbb H_{h,i}}.
\end{equation}
\end{lemma}

\begin{proof}[Proof of Lemma \ref{lem:RKHS:X}]

For $h=1,\dots,H$ and for $t\in [-1,1]^{d_{\max}}$, recall that $T_h(t)\in [-1,1]^{d_h}$ is composed of the $d_h$ first coordinates of $t$. 
For $(h,i) \in \mathcal{I}$, let $\mathbb H_{0,h,i}$ be the set of functions of the form
\begin{equation} \label{eq:generic:H:zero:hi}
t \in 
[-1,1]^{d_h}
\mapsto
\sum_{a=1}^{N_{h,i}} 
\alpha_{h,i,a}
 k_{h,i} (  t_{h,i,a}, t ),
\end{equation}
for $N_{h,i} \in \mathbb{N}^*$, for $t_{h,i,1},\dots,t_{h,i,N_{h,i}} \in [-1,1]^{d_h}$, and for $\alpha_{h,i,1} , \dots , \alpha_{h,i,N_{h,i}} \in \mathbb{R}$. 

Also, let $\mathbb H_{0}$ be the set of functions of the form
\begin{equation} \label{eq:bivariate:g}
(t,h,i) \in \mathcal{X}
\mapsto
\sum_{a=1}^{N} 
\gamma_a k( (t_a,h_a,i_a),(t,h,i) ),
\end{equation}
for $N \in \mathbb{N}^*$, for $(t_1,h_1,i_1),\dots,(t_N,h_N,i_N) \in \mathcal{X}$, and for $\gamma_1 , \dots , \gamma_N \in \mathbb{R}$. 
Then, from Moore-Aronszajn theorem (see \cite[Theorem 3]{berlinet2004reproducing}), we have $\mathbb H_{0,h,i} \subset \mathbb H_{h,i}$ for $(h,i) \in \mathcal{I}$ and $\mathbb H_0 \subset \mathbb H$.

We have, for any function $g_0 \in \mathbb H_0$ of the form  \eqref{eq:bivariate:g}, using the independence of the $Z_{h,i}$ for $(h,i) \in \mathcal{I}$,
that $P_{h,i} 
(g_0)$ is the function
\[
t \in [-1,1]^{d_h}
\mapsto 
\sum_{\substack{a=1 \\ 
		(h_a,i_a)=(h,i)}}^{N}
\gamma_a
k_{h,i}(  T_{h,i}( t_a), t)
\]
and thus belongs to $\mathbb H_{0,h,i}$. 

Hence we have, using the independence of the $Z_{h,i}$ for $(h,i) \in \mathcal{I}$, again,
\begin{align} \label{eq:relating:RKHS:norms}
\norme{g_0}^2_{\mathbb H}
 = &
 \sum_{a,b=1}^{N}
 \gamma_a \gamma_b
 k( (t_a,h_a,i_a),(t_b,h_b,i_b) )) 
 \notag
 \\
 = &
 \sum_{(h,i) \in \mathcal{I}} 
  \sum_{\substack{a,b=1 \\ 
  (h_a,i_a)=(h_b,i_b)=(h,i)}}^N \gamma_a \gamma_b
 k_{h,i}(  T_{h_a}( t_a), T_{h_b}(t_b))
  \notag
 \\
 = &
  \sum_{(h,i) \in \mathcal{I}} 
\norme{
P_{h,i} 
(g_0)
}^2_{\mathbb H_{h,i}}.
\end{align}

Now let $g \in \mathbb H$. Again from Moore-Aronszajn theorem, there exists a sequence $(g_{N})_{{N} \in \mathbb{N}^*}$ of elements of $\mathbb H_0$, that is a Cauchy sequence with $\norme{\cdot{}}_{\mathbb H}$ converging pointwise to $g$.
This implies that the first property in \eqref{eq:structure:RKHS:X} holds for $g$ since it holds for $g_N$ of the form \eqref{eq:bivariate:g}. 
 For $(h,i) \in \mathcal{I}$, let $g_{{N},h,i} = P_{h,i} (g_N) $. From \eqref{eq:relating:RKHS:norms} and the linearity of $P_{h,i}$, 
 $(g_{{N},h,i})_{{N} \in \mathbb{N}^*}$  is a Cauchy sequence of elements of $\mathbb H_{0,h,i}$ with the norm $\norme{ \cdot }_{\mathbb H_{h,i}}$. Hence, by completeness, as ${N} \to \infty$, $g_{N,h,i}$ converges (with the norm $\norme{ \cdot }_{\mathbb H_{h,i}}$) to a function $g_{h,i} \in \mathbb H_{h,i}$.

Since RKHS convergence implies pointwise convergence \cite[Lemma 8]{berlinet2004reproducing}, we have, for $t \in [-1,1]^{d_h}$,
\begin{align*}
P_{h,i}(g)(t)
=  \lim_{{N} \to \infty}  P_{h,i} ( g_N)(t) 
=  \lim_{{N} \to \infty} g_{N,h,i}(t) 
 = g_{h,i}(t).
\end{align*}
 Hence the second property in \eqref{eq:structure:RKHS:X} holds and $g$ can be written as in \eqref{eq:structure:RKHS:X}.
Conversely, let a function $g : \mathcal{X} \to \mathbb{R}$ that satisfies \eqref{eq:structure:RKHS:X}. For $(h,i) \in \mathcal{I}$,
from Moore-Aronszajn theorem, there exist $(g_{N,h,i})_{N \in \mathbb{N}^*}$ as in \eqref{eq:generic:H:zero:hi} that is a Cauchy sequence of elements of $\mathbb H_{0,h,i}$ with the norm $\norme{\cdot{}}_{\mathbb H_{h,i}}$ and such that, as $N \to \infty$, $g_{N,h,i} \to P_{h,i}(g)$ pointwise
on $[-1,1]^{d_h}$. 
Then let us show that the function $g_{N}: \mathcal{X} \to \mathbb{R}$ defined by, for 
$ (t,h, i) \in \mathcal{X} $,
\begin{equation} \label{eq:def:gN}
g_{N}(t,h, i) =
g_{N,h,i}
\left( 
T_{h}(t)
\right),
\end{equation}
belongs to $\mathbb H_0$. We have
\begin{align*}
g_{N}(t,h, i)
= &
\sum_{a=1}^{N_{h,i}} 
\alpha_{h,i,a}
k_{h,i} (  t_{h,i,a}, T_{h}(t) )
\\
= &
\sum_{a=1}^{N_{h,i}} 
\alpha_{h,i,a}
k(  ((t_{h,i,a},0),h,i),(t,h, i) )
\\
= &
\sum_{(h',i') \in \mathcal{I}}
\sum_{a=1}^{N_{h',i'}} 
\alpha_{h',i',a}
k(  ((t_{h',i',a},0),h',i'),(t,h, i)),
\end{align*}
using the independence of the $P_{h,i}(W)$ for $(h,i) \in \mathcal{I}$. Hence, $g_N$ belongs to $\mathbb H_{0,N}$ and thus to $\mathbb H_0$.

Observe that we have, for $(h,i) \in \mathcal{I}$, $P_{h,i}(g_N) = g_{N,h,i}$ and thus from \eqref{eq:relating:RKHS:norms}, as $N \to \infty$,
$g_{N}$ is a Cauchy sequence of elements in  $\mathbb H_0$ with the norm $\norme{\cdot{}}_{\mathbb H}$. 
Hence  $g_{N}$ also converges pointwise on $\mathcal{X}$ and the pointwise limit function $\lim_{N \to \infty}
g_{N}$ belongs to $\mathbb{H}$ from Moore-Aronszajn theorem.
We also have, for fixed $(t,h, i) \in \mathcal{X}$,
\begin{align*}
	\lim_{N \to \infty} g_N(t,h, i)
	&= 
		\lim_{N \to \infty}
		g_{N,h,i}
		\left( 
		T_h(t)
		\right)
		\\&=
		P_{h,i}(g)
		\left( 
		T_h(t)
		\right) 
		=
		g(t,h,i),
\end{align*}
using at the end the first equality in \eqref{eq:structure:RKHS:X}.
Hence, any $g : \mathcal{X} \to \mathbb{R}$ satisfying \eqref{eq:structure:RKHS:X} does belong to $\mathbb H$. This shows that the RKHS $\mathbb H$ is as indicated in the lemma.

Finally, let us prove \eqref{eq:norm:equality:bivariate}.
Let $g : \mathcal{X} \to \mathbb{R}$ in $\mathbb{H}$ and thus satisfying \eqref{eq:structure:RKHS:X} . Then, for $(h,i) \in \mathcal{I}$,
there exists a sequence $(g_{N,h,i})_{N \in \mathbb{N}^*}$  of functions in $\mathbb H_{0,h,i}$ 
that converges to $P_{h,i}(g)$ 
with the norm $\norme{\cdot{}}_{\mathbb H_{h,i}}$.
 Define $g_{N}$ as in \eqref{eq:def:gN}. We have shown that $(g_{N})_{N \in \mathbb{N}^*}$ is a Cauchy sequence of elements in  $\mathbb H_0$ with the norm $\norme{\cdot{}}_{\mathbb H}$ that converges pointwise to $g$.
  Hence $g_{N}$ also converges with the norm $\norme{\cdot{}}_{\mathbb H}$ to $g$ and we have, using \eqref{eq:relating:RKHS:norms},
\begin{align*}
\norme{g}^2_{\mathbb H}
= &
\lim_{N \to \infty}
\norme{g_{N} }^2_{\mathbb H}
= 
\lim_{N \to \infty}
\sum_{(h,i) \in \mathcal{I}}
\norme{
P_{h,i} (g_N)
 }^2_{\mathbb H_{h,i}}.
\end{align*}
Because of \eqref{eq:def:gN}, we have $P_{h,i} (g_N) = g_{N,h,i}$ and thus
\begin{align*}
	\norme{g}^2_{\mathbb H}
	= &
	\lim_{N \to \infty}
	\sum_{(h,i) \in \mathcal{I}}
	\norme{
		g_{N,h,i}
	}^2_{\mathbb H_{h,i}}
	=
		\sum_{(h,i) \in \mathcal{I}}
	\norme{
P_{h,i}(g)
}^2_{\mathbb H_{h,i}}.
\end{align*}
This concludes the proof.
\end{proof}

\begin{lemma} \label{lem:decentering:small:balls}
Consider a centered Gaussian process $X$, valued in a Banach space $(E , \norme{\cdot})$ composed of functions from a set $T$ to $\mathbb{R}$. Let $\mathbb H_X$ be the RKHS of $X$, with RKHS-norm $\norme{\cdot}_{\mathbb H_X}$.  Let $f \in E$. Then we have, for any $\varepsilon >0$, 
\[
- \log \mathbb{P}
\left(
\norme{ X -f }
 < 2 \varepsilon
\right)
\leqslant 
\frac 12
\inf_{
\substack{
g \in \mathbb H_X \\
\norme{ g - f } < \varepsilon
}
}
 \norme{g}_{\mathbb H_X}^2 
- \log \mathbb{P}
\left(
\norme{ X  }
 \leqslant  \varepsilon
\right),
\]
with the convention that the infimum above is equal to $ + \infty$ if there are no $g \in \mathbb H_X$ such that $\norme{ g - f } < \varepsilon$.  
\end{lemma}

\begin{proof}[Proof of Lemma \ref{lem:decentering:small:balls}]
If there is no $g \in \mathbb H_X$ such that $\norme{ g - f } < \varepsilon$, then the inequality of the lemma is trivially true.
If there exists one $g \in \mathbb H_X$ with $\norme{ g - f } < \varepsilon$, we have, for such a $g$, using first the triangle inequality and then \cite[Theorem 3.1]{li2001gaussian},
\begin{align*}
 - \log \mathbb{P}
\left(
\norme{ X -f }
 < 2 \varepsilon
\right)
& \leqslant 
-
\log \mathbb{P}
\left(
\norme{ X -g }
 \leqslant \varepsilon
\right)
\leqslant 
\frac{1}{2} \norme{g}_{H_X}^2 
- \log \mathbb{P}
\left(
\norme{ X  }
 <  \varepsilon
\right).
\end{align*}
This concludes the proof. 
\end{proof}

\begin{lemma} \label{lemma:proba:small:ball:non-zero}
Let $T$ be a compact metric set and let $X$ be a continuous centered univariate Gaussian process indexed by $T$. Then, for $u >0$, 
\[
\mathbb{P}
\Bigl(
\sup_{t \in T}
\abs{ X(t) }
\leqslant
u
\Bigr)
> 0.
\]
\end{lemma}

\begin{proof}[Proof of Lemma \ref{lemma:proba:small:ball:non-zero}]
This result is often stated implicitly in the literature (for instance in \cite{li2001gaussian}) but we are not aware of an explicit statement. Here is a proof.

Let $u >0$ and let $(t_i)_{i \in \mathbb{N}^*}$ be a sequence of elements in $T$ that is dense in $T$.
For $n \in \mathbb{N}^*$, let $\mathcal{F}_n$ be the $\sigma$-algebra generated by $X(t_1), \dots , X(t_n)$.
 Then, from \cite[Proposition 2.9]{bect2019supermartingale} for instance, $\mathbb{E}[ X(\cdot) | \mathcal{F}_n ] \to X$ uniformly on $T$ and almost surely. Hence, one has
 \[
 \mathbb{P}
 \left(
 \sup_{t \in T}
 \abs{ X(t) - \mathbb{E}[ X(t) | \mathcal{F}_n ] }
 \leqslant \frac{u}{2}
 \right)
 \to 1
 \]
 as $n \to \infty$. Using the triangular inequality and from the independence between $X - \mathbb{E}[ X( \cdot ) | \mathcal{F}_n ]$ and $\mathbb{E}[ X( \cdot ) | \mathcal{F}_n]$ (as the conditional distribution of the first process given $\mathcal{F}_n$ is deterministic), we obtain
 \begin{align*}
 \mathbb{P}
\left(
\sup_{t \in T}
\abs{ X(t) }
\leqslant 
u
\right)
& \geqslant 
 \mathbb{P}
\left(
\sup_{t \in T}
 \abs{ X(t) - \mathbb{E}( X(t) | \mathcal{F}_n ) }
 \leqslant \frac{u}{2}
 ,
 \sup_{t \in T}
 \abs{\mathbb{E}( X(t) | \mathcal{F}_n ) }
 \leqslant \frac{u}{2}
\right)
\\
& = 
 \mathbb{P}
\left(
\sup_{t \in T}
 \abs{ X(t) - \mathbb{E}( X(t) | \mathcal{F}_n ) }
 \leqslant \frac{u}{2}
\right)
 \mathbb{P}
\left(
 \sup_{t \in T}
 \abs{\mathbb{E}( X(t) | \mathcal{F}_n ) }
 \leqslant \frac{u}{2}
\right).
 \end{align*}
The first probability above is non-zero for $n$ large enough, as seen before. The second probability is non-zero for any $n \in \mathbb{N}^*$, because $\sup_{t \in T}
 \abs{\mathbb{E}( X(t) | \mathcal{F}_n ) }$ is a continuous function of the Gaussian vector $(X(t_1), \dots , X(t_n))$, since the covariance function of $X$ is continuous \cite[Lemma 1]{ibragimov78gaussian}.
\end{proof}

\begin{lemma} \label{lem:proba:maximum:equals}
Let $X$ be a centered univariate Gaussian process with continuous realizations, indexed by a compact metric  set $T$. Then, for all $u >0$, 
\[
\mathbb{P}\left( \displaystyle\sup_{t \in T} \abs{X(t)} = u \right) =0.
\] 
In other words, the random variable $\displaystyle\sup_{t \in T} \abs{X(t)}$ is absolutely continuous on the half-line $(0,+\infty)$.
\end{lemma}

\begin{proof}[Proof of Lemma \ref{lem:proba:maximum:equals}]

Let $T'$ be a dense countable subset of $T$. By density, since $T'$ is dense, and $X$ has continuous realizations, we have, for $u >0$,
\begin{align*}
\mathbb{P}( \sup_{t \in T} \abs{X(t)} = u )
& = 
\mathbb{P}( \sup_{t \in T'} \abs{X(t)} = u )
\\
& \leqslant 
\mathbb{P}( \sup_{t \in T'} X(t) = u )
+
\mathbb{P}( \inf_{t \in T'} X(t) = - u )
\\
& = 
\mathbb{P}( \sup_{t \in T'} X(t) = u )
+
\mathbb{P}( \sup_{t \in T'} ( -  X(t)) = u )
\\
& = 2
\mathbb{P}( \sup_{t \in T'} X(t) = u ),
\end{align*}
where we have used the symmetry of the law of a centered Gaussian process. This last probability is zero from Tsirelson's theorem (\cite{tsirelson76density}, see also \cite[Theorem 7.1]{AW09}), together with Lemma \ref{lemma:proba:small:ball:non-zero}.
\end{proof}

\begin{lemma} \label{lem:constraint:function:derivative}
Let $X$ be a centered univariate Gaussian process indexed by $[-1,1]^\ell$ for some $\ell \in \mathbb{N}^*$ with continuously differentiable realizations. Then, for $a >0$ and for $b_1 , \dots , b_\ell >0$, we have\begin{align*}
	&
\mathbb{P}
\Bigl(
\sup_{t \in [-1,1]^\ell}
\abs{ X(t) }
\leqslant a
,
\sup_{t \in [-1,1]^\ell}
\abs{ \partial X (t) / \partial x_i }
\leqslant b_i \text{ for }
i = 1 ,\dots, \ell
\Bigr)
\\
& \qquad
\geqslant 
\mathbb{P}
\Bigl(
\sup_{t \in [-1,1]^\ell}
\abs{ X(t) }
\leqslant a
\Bigr)
\prod_{i=1}^\ell
\mathbb{P}
\Bigl(
\sup_{t \in [-1,1]^\ell}
\abs{ \partial X (t) / \partial x_i }
\leqslant b_i
\Bigr).
\end{align*}
\end{lemma}

\begin{proof}[Proof of Lemma \ref{lem:constraint:function:derivative}]
Let $T = [-1,1]^\ell$ and  $(t_j)_{j \in \mathbb{N}^*}$ be a sequence of elements of $T$ that is dense in $T$. 
From the dominated convergence theorem and Lemma \ref{lem:proba:maximum:equals}, as $n \to \infty$,
\begin{align*}
	&
\mathbb{P}
\left(
\max_{j=1 , \dots , n}
\abs{ X(t_j) }
\leqslant a
,
\max_{j=1 , \dots , n}
\abs{ \partial X(t_j) / \partial x_i }
\leqslant b_i  \text{ for }
i = 1 , \dots , \ell
\right)
\\
& \qquad 
\to 
\mathbb{P}
\Bigl(
\sup_{t \in T}
\abs{ X(t) }
\leqslant a
,
\sup_{t \in T}
\abs{ \partial X (t) / \partial x_i }
\leqslant b_i  \text{ for }
i = 1 ,\dots, \ell
\Bigr).
\end{align*}
Let 
\[
A = \left\{
  (z_1,\dots,z_n,z^{(1)}_1 ,\dots,z^{(1)}_n, \dots,z^{(\ell)}_1 ,\dots,z^{(\ell)}_n) \in \mathbb{R}^{n + \ell n} ; \max_{j=1 , \dots , n}
\abs{ z_j } \leqslant a 
 \right\}
\] 
and, for $i  =1 , \dots , \ell$
\[
B_i = 
\left\{ (z_1,\dots,z_n,z^{(1)}_1 ,\dots,z^{(1)}_n, \dots,z^{(\ell)}_1 ,\dots,z^{(\ell)}_n) \in \mathbb{R}^{n + \ell n} ; \max_{j=1 , \dots , n}
\abs{ z^{(i)}_j } \leqslant b_i  
\right\}.
\]  
Then $A$ and $B$ are convex symmetric subsets of $\mathbb{R}^{2n}$. Hence, from the Gaussian correlation inequality (\cite{royen14simple}, see also \cite{latala2017royen}), we obtain
\begin{align*} 
& \mathbb{P}
\left(
\max_{j=1 , \dots , n}
\abs{ X(t_j) }
\leqslant a
,
\max_{j=1 , \dots , n}
\abs{ \partial X(t_j) / \partial x_i }
\leqslant b_i,
i = 1 , \dots , \ell
\right)
\\
& = 
\mathbb{P}
\Big(
\big( X(t_1),\dots,X(t_n),
\partial X(t_1) / \partial x_1,\dots, \partial X(t_n) / \partial x_1 
, \dots ,
\partial X(t_1) / \partial x_\ell,\dots, \partial X(t_n) / \partial x_\ell
\big) 
\\
& \quad 
\in A \cap \cap_{i=1}^\ell B_i
\Big)
\\
& \geqslant 
\mathbb{P}
\Big(
\big(
X(t_1),\dots,X(t_n),
\partial X(t_1) / \partial x_1,\dots, \partial X(t_n) / \partial x_1 
, \dots ,
\partial X(t_1) / \partial x_\ell,\dots, \partial X(t_n) / \partial x_\ell
\big)
\in A 
\Big) \\
& \quad \times 
\prod_{i=1}^\ell
\mathbb{P}
\Big( 
\big(
X(t_1),\dots,X(t_n),
\partial X(t_1) / \partial x_1,\dots, \partial X(t_n) / \partial x_1 
, \dots ,
\partial X(t_1) / \partial x_\ell,\dots, \partial X(t_n) / \partial x_\ell
\big)
\in  B_i
\Big)
\\
& = 
\mathbb{P}
\bigl(
\max_{j=1 , \dots , n}
\abs{ X(t_j) }
\leqslant a
\bigr)
\times \prod_{i=1}^\ell \P\Big( \max_{j=1,\dots,n} \abs{ \partial X(t_j)/\partial x_i}\leqslant b_i\Big)
\\
& 
\to 
\mathbb{P}
\Bigl(
\sup_{t \in T}
\abs{ X(t) }
\leqslant a
\Bigr)\times
\prod_{i=1}^\ell
\mathbb{P}
\Bigl(
\sup_{t \in T}
\abs{ \partial X (t) / \partial x_i }
\leqslant b_i
\Bigr),
\end{align*}
as $n \to \infty$, again from the dominated convergence theorem and Lemma \ref{lem:proba:maximum:equals}. This concludes the proof.
\end{proof}

Now we are able to prove  Theorem \ref{thm:analyzing:phi}.

\begin{proof}[Proof of Theorem \ref{thm:analyzing:phi}]

From the assumptions of the theorem and 
Lemma \ref{lem:RKHS:X}, we obtain that $w_0$ is indeed in the closure of $\mathbb H$ in $\mathbb{B}$ (with respect to $\|\cdot{}\|_{\infty}$).
Then, using  Lemma \ref{lem:RKHS:X}, $\varepsilon \in (0,1]$ and
\eqref{eq:analyzing:phi:cond:un},  we obtain
\begin{align}
& \phi_{c,w_0}(\varepsilon)  \notag
=  \underset{
\substack{ g \in \mathbb H \\
\norme{ g - w_0 }_{\infty} < \varepsilon 
} 
}{\inf}
\norme{g}_{\mathbb H}^2
-
\log
\mathbb{P}
\left(
\norme{ W_c }_{\infty} 
< \varepsilon
\right) -\log \P(\norme{W -w_0}_{\infty}<2\varepsilon,\ W\in \mathbb B_c) \notag
\\
 \leqslant  &
\underset{
\substack{ 
\{ g_{h,i} ; (h,i) \in \mathcal{I}  \}	\\
\text{for all} ~ (h,i), g_{h,i} \in \mathbb H_{h,i}, \\
\norme{ g_{h,i} - P_{h,i} (w_0) }_{\infty} < \varepsilon  
} 
}{\inf}
\left(
\sum_{(h , i ) \in \mathcal{I}} 
\norme{g_{h,i}}_{\mathbb H_{h,i}}^2
\right) \notag
\\
&
-
\log \mathbb{P}
\Big(
\norme{Z_{h,i}}_{\infty} < \varepsilon \text{ for } (h,i) \in \mathcal{I}
\text{ and } \notag\\
& \qquad \qquad
\norme{ \partial Z_{h,i} / \partial x_j }_{\infty} \leqslant K_{h,i,j} \text{ for } 
h = 2,\dots, H,
 \text{ for }  i=1,\dots,d_{h+1},
 \text{ for }  j = 1 , \dots , d_{h} 
\Big) \notag
\\
&
-
\log \mathbb{P}
\Big(
\norme{Z_{h,i} - P_{h,i}(w_0)}_{\infty} < 2 \varepsilon  
\text{ for } (h,i) \in \mathcal{I} \text{ and }
\notag
\\
& 
\qquad \qquad \norme{ \partial Z_{h,i} / \partial x_j }_{\infty} \leqslant K_{h,i,j} \text{ for }
h = 2,\dots, H, \text{ for } i=1,\dots,d_{h+1}, \text{ for }
j = 1 , \dots , d_{h} \Big). \notag
\end{align}

Using now the independence between the $Z_{h,i}$'s for $(h,i) \in \mathcal{I}$, we obtain
\begin{equation} \label{eq:bounding:phi:c}
\phi_{c,w_0}(\varepsilon)
\leqslant 
\sum_{
i = 1}^{d_2 
}
A_{i}
+
\sum_{
	\substack{
(h,i) \in \mathcal{I} \\
h \geqslant 2
}
}
B_{h,i},
\end{equation}
where we have defined
\[
A_i = 
\inf_{
	\substack{
 g \in \mathbb H_{1,i} \\
 \norme{  g - P_{1,i}(w_0) }_{\infty} < \varepsilon
}
}
\norme{  g }_{\mathbb H_{1,i}}^2
-  \log 
\mathbb{P}
\left(
\norme{Z_{1,i}}_{\infty} < \varepsilon 
\right)
-
 \log 
\mathbb{P}
\left(
\norme{Z_{1,i} - P_{1,i}(w_0)  }_{\infty} < 2 \varepsilon 
\right)
\]
and 
\begin{align*}
B_{h,i}  = &
\inf_{
	\substack{
		g \in \mathbb H_{h,i} \\
		\norme{  g - P_{h,i}(w_0) }_{\infty} < \varepsilon
	}
}
\norme{  g }_{\mathbb H_{h,i}}^2\\
&
-   \log 
\mathbb{P}
\Big(
\norme{Z_{h,i}}_{\infty} < \varepsilon ,
\norme{ \partial Z_{h,i} / \partial x_j }_{\infty} \leqslant K_{h,i,j}  \text{ for }
j = 1 , \dots , d_{h}  
\Big)
\\
&
-   \log 
\mathbb{P}
\Big(
\norme{Z_{h,i} - P_{h,i}(w_0)  }_{\infty} < 2 \varepsilon,
\norme{ \partial Z_{h,i} / \partial x_j }_{\infty} \leqslant K_{h,i,j} \text{ for } 
j = 1 , \dots , d_{h}  
\Big).
\end{align*}

First let us bound $A_i$. 
Since $P_{1,i}(w_0)$ is in the closure of $H_{1,i}$ in $(\mathcal{C}^0( [-1,1]^{d_1} ) , \mathbb{R}) , \norme{\cdot}_{\infty})$, using Lemmas \ref{lem:decentering:small:balls}, \ref{lemma:proba:small:ball:non-zero}, and \ref{lem:proba:maximum:equals}, we have

\begin{align} \label{eq:bound:A:un}
	A_i 
	& \leqslant 
		\inf_{
		\substack{
			g \in \mathbb H_{1,i} \\
			\norme{  g - P_{1,i}(w_0) }_{\infty} < \varepsilon
		}
	}
	\norme{  g }_{\mathbb H_{1,i}}^2
	-  \log 
	\mathbb{P}
	\left(
	\norme{Z_{1,i}}_{\infty} < \varepsilon 
	\right) \notag
	+ 
	\frac{1}{2}
		\inf_{
		\substack{
			g \in \mathbb H_{1,i} \\
			\norme{  g - P_{1,i}(w_0) }_{\infty} < \varepsilon
		}
	}
	\norme{  g }_{\mathbb H_{1,i}}^2
-  \log 
\mathbb{P}
\left(
\norme{Z_{1,i}}_{\infty} < \varepsilon 
\right)   \notag
\\
&  =
\frac{3}{2} 	\inf_{
	\substack{
		g \in \mathbb H_{1,i} \\
		\norme{  g - P_{1,i}(w_0) }_{\infty} < \varepsilon
	}
}
	\norme{  g }_{\mathbb H_{1,i}}^2 
- 2 \log 
\mathbb{P}
\left(
\norme{Z_{1,i}}_{\infty} < \varepsilon 
\right). 
\end{align}

Second let us bound $B_{h,i}$. 
 We have, from \eqref{eq:analyzing:phi:cond:deux},
\begin{align} \label{eq:bounding:Bhi}
&	B_{h,i}  \leqslant
	\inf_{
		\substack{
			g \in \mathbb H_{h,i} \\
			\norme{  g - P_{h,i}(w_0) }_{\infty} < \varepsilon
		}
	}
	\norme{  g }_{\mathbb H_{h,i}}^2
	-   \log 
	\mathbb{P}
	\Big(
	\norme{Z_{h,i}}_{\infty} < \varepsilon ,
	\norme{ \partial Z_{h,i} / \partial x_j }_{\infty} \leqslant K_{\min}, 
	j = 1 , \dots , d_{h}  
	\Big)
	\notag
	\\
	&\quad 
	-   \log 
	\mathbb{P}
	\Big(
	\norme{Z_{h,i} - P_{h,i}(w_0)  }_{\infty} < 2 \varepsilon,
	\norme{ \partial Z_{h,i} / \partial x_j
- 	
\partial P_{h,i}(w_0) / \partial x_j
 }_{\infty} \leqslant K_{\min}/2, 
	j = 1 , \dots , d_{h}  
	\Big)
	\notag
	\\
	& \leqslant 
		\inf_{
		\substack{
			g \in \mathbb H_{h,i} \\
			\norme{  g - P_{h,i}(w_0) }_{\infty} < \varepsilon
		}
	}
	\norme{  g }_{\mathbb H_{h,i}}^2
	-   \log 
	\mathbb{P}
	\Big(
	\norme{Z_{h,i}}_{\infty} < \varepsilon ,
	\norme{ \partial Z_{h,i} / \partial x_j }_{\infty} \leqslant K_{\min}, 
	j = 1 , \dots , d_{h}  
	\Big)\notag
	\\
	& \quad
	-   \log 
	\mathbb{P}
	\Big(
	\norme{Z_{h,i} - P_{h,i}(w_0)  }_{\infty}  \leqslant   \varepsilon,
	2 \varepsilon
	\norme{ \partial Z_{h,i} / \partial x_j
		- 	
		\partial P_{h,i}(w_0) / \partial x_j
	}_{\infty}/K_{\min}  \leqslant \varepsilon, 
	j = 1 , \dots , d_{h}  
	\Big).
\end{align}

Let us define the Banach space $\mathbb{B}_{h}$ as the set of continuous functions from $[-1,1]^{d_h} \times \{ 0 ,  1 , \dots , d_{h}  \}$,
equipped with the norm $\norme{ \cdot }_{\infty,\varepsilon}$ defined by, for $z \in \mathbb{B}_h$, 
\[
\norme{ z }_{\infty,\varepsilon}
=
\norme{ z(\cdot , 0) }_{\infty}
\vee
\frac{ 2 \varepsilon}{ K_{\min}} 
\norme{ z(\cdot , 1) }_{\infty} 
\vee 
\dots 
\vee 
\frac{ 2 \varepsilon}{ K_{\min}} 
\norme{ z(\cdot , d_{h}) }_{\infty} .
\]

We consider the map $M$ from $(\mathcal{C}^1( [-1,1]^{d_h} , \mathbb{R} ) , \norme{\cdot}_{\infty,1})$ to $(\mathbb{B}_h,\norme{\cdot}_{\infty,\varepsilon})$ defined by, for $f \in \mathcal{C}^1( [-1,1]^{d_h} , \mathbb{R})$,  $(Mf)( \cdot,0 ) = f$, $(Mf)(  \cdot , 1 ) = \partial f / \partial x_1 $, ... ,$(Mf)(  \cdot , d_h ) = \partial f / \partial x_{d_h} $. 
Let us define the Gaussian process $\overline{Z}$ on $\mathbb{B}_h$ by $\overline{Z} = M (Z_{h,i})$and similarly the function $\overline{w}_0$ in $\mathbb{B}_h$ by $\overline{w}_0 = M( P_{h,i}(w_0))$.
We write  $\mathbb H_{\overline{Z}}$ the RKHS of the Gaussian process $\overline{Z}$ and 
  $\norme{\cdot}_{\mathbb H_{\overline{Z}}}$ its RKHS-norm.

\begin{lemma} \label{lem:in:RKHS}
Recall that $h= 2,\dots,H$. We have $\mathbb H_{\overline{Z}} = \{ M (g) ; g \in  \mathbb H_{h,i}\}$ and, for $g \in \mathbb H_{h,i}$, $\norme{g}_{\mathbb H_{h,i}} = \norme{M(g)}_{H_{\overline{Z}}}$. 
\end{lemma}

The proof of Lemma \ref{lem:in:RKHS} is postponed after the proof of Theorem \ref{thm:analyzing:phi}.
Then, using the previous notation and Lemmas \ref{lem:decentering:small:balls} and  \ref{lem:in:RKHS}, we obtain
\begin{align*}
	&
	-   \log 
\mathbb{P}
\Big(
\norme{Z_{h,i} - P_{h,i}(w_0)  }_{\infty}  \leqslant   \varepsilon,
\frac{2 \varepsilon}{K_{\min}}
\norme{ \partial Z_{h,i} / \partial x_j
	- 	
	\partial P_{h,i}(w_0) / \partial x_j
}_{\infty} \leqslant \varepsilon \text{ for } 
j = 1 , \dots , d_{h}  
\Big)\\
	& =  -
	\log \mathbb{P}
	\left(
	\norme{ 
		\overline{Z} - \overline{w}_0
	}_{\infty, \varepsilon} \leqslant  \varepsilon
	\right)\notag\\ 
	& \leqslant  \frac 12
	\underset{
		\substack{
			h \in \mathbb H_{\overline{Z}} \\
			\norme{ h - \overline{w}_0 }_{\infty, \varepsilon} < \frac{\varepsilon}{2}
		} 
	}{\inf}
	\norme{h}_{\mathbb H_{\overline{Z}}}^2
	- \log \mathbb{P}
	\left(
	\norme{ 
		\overline{Z}
	}_{\infty, \varepsilon} \leqslant  \frac{\varepsilon}{2}
	\right)
	\notag
	\\
	& = \frac 12
	\underset{
		\substack{
			g \in \mathbb H_{h,i} \\
			\norme{ g - P_{h,i}(w_0) }_{\infty} < \frac{\varepsilon}{2} \\
			\norme{ \partial g / \partial x_j - \partial P_{h,i} (w_0) / \partial x_j }_{\infty} < \frac{K_{\min}}{4}, \\
			j = 1 , \dots , d_h
		} 
	}{\inf}
	\norme{g}_{\mathbb H_{h,i}}^2 \notag
	- \log \mathbb{P}
	\Bigl(
	\norme{ 
		Z_{h,i}
	}_{\infty} \leqslant  \frac{\varepsilon}{2}
	,
	\norme{ 
		\partial Z_{h,i} / \partial x_j
	}_{\infty}
	\leqslant  \frac{K_{\min}}{4} \text{ for }
	j = 1, \dots , d_h
	\Bigr).
\end{align*}

Now using the above display and \eqref{eq:bounding:Bhi}, together with Lemmas  \ref{lem:proba:maximum:equals}, and \ref{lem:constraint:function:derivative}, we obtain
\begin{align*}
B_{h,i}
&
 \leqslant 
 \frac{3}{2}
 \underset{
\substack{
g \in \mathbb H_{h,i} \\
\norme{ g - P_{h,i}(w_0) }_{\infty} < \frac{\varepsilon}{2} \\
\norme{ \partial g / \partial x_j - \partial P_{h,i}(w_0) / \partial x_j }_{\infty} < \frac{K_{\min}}{4}, \\
j = 1 , \dots , d_h
} 
}{\inf}
\norme{g}_{\mathbb H_{h,i}}^2
- 2 \log \mathbb{P}
\Bigl(
\norme{ 
Z_{h,i}
 }_{\infty} \leqslant  \frac{\varepsilon}{2}
\Bigr)
\\
&
\quad - 2 
\sum_{j = 1}^{d_h}
\log \mathbb{P}
\Bigl(
 \norme{ 
\partial Z_{h,i} / \partial x_j
 }_{\infty}
  \leqslant  \frac{K_{\min}}{4}
\Bigr).
\end{align*}
Then the proof is concluded using \eqref{eq:bounding:phi:c} and \eqref{eq:bound:A:un}. 
\end{proof}

It remains to prove Lemma \ref{lem:in:RKHS}.

\begin{proof}[Proof of Lemma \ref{lem:in:RKHS}]
By assumption, the Gaussian process $Z_{h,i}$, with input space $[-1,1]^{d_h}$, has continuously differentiable realizations. Hence, $Z_{h,i}$ can be viewed as a random element of the Banach space $\mathcal{C}^0( [-1,1]^{d_h} , \mathbb{R} )$ of continuous functions on
 $[-1,1]^{d_h}$, equipped with the norm $\norme{\cdot}_{\infty}$. 
The space $(\mathbb H_{h,i},\norme{\cdot}_{\mathbb H_{h,i}})$ is defined as in  \cite[Section 2.1]{van2008reproducing} (from the covariance function of $Z_{h,i}$). From \cite[Theorem 2.1]{van2008reproducing}, $\mathbb H_{h,i}$ can be equivalently defined as in  \cite[Section 2.2]{van2008reproducing} from the law of $Z_{h,i}$ in $\mathcal{C}^0( [-1,1]^{d_h} , \mathbb{R} )$ and from the norm $\norme{\cdot}_{\infty}$.

The identity map from $(\mathcal{C}^1( [-1,1]^{d_h} , \mathbb{R} ) , \norme{\cdot}_{\infty,1})$ to $(\mathcal{C}^0( [-1,1]^{d_h} , \mathbb{R} ) , \norme{\cdot}_{\infty})$ is a one-to-one, continuous and linear application. Hence, from \cite[Lemma 7.1]{van2008reproducing}, the RKHS $\mathbb H_{h,i}$ can also be equivalently defined as in  \cite[Section 2.2]{van2008reproducing}, but this time by seeing $Z_{h,i}$ as a random element of $(\mathcal{C}^1( [-1,1]^{d_h} , \mathbb{R} ) , \norme{\cdot}_{\infty,1})$.  

Now we consider the Banach space $\mathbb{B}_h$ equipped with the norm $\norme{\cdot}_{\infty,\varepsilon}$. We consider the map $M$ introduced above. Then $M$ is a one-to-one, continuous and linear application  from $(\mathcal{C}^1( [-1,1]^{d_h} , \mathbb{R} ) , \norme{\cdot}_{\infty,1})$ to $(\mathbb{B}_h , \norme{\cdot}_{\infty,\varepsilon})$.
Also recall that we have defined $\overline{Z} = M (Z_{h,i})$.
 Hence, from \cite[Lemma 7.1]{van2008reproducing}, the RKHS  $\mathbb H_{\overline{Z}}$ of the Gaussian process $\overline{Z}$ (defined as in \cite[Section 2.2]{van2008reproducing}, by seeing $\overline{Z}$ as a random element of $(\mathbb{B}_h , \norme{\cdot}_{\infty,\varepsilon})$), is equal to $\{M (g) ; g \in\mathbb H_{h,i}\}$. Furthermore, for $g \in \mathbb H_{h,i}$, we have $\norme{g}_{\mathbb H_{h,i}} = \norme{M(g)}_{\mathbb H_{\overline{Z}}}$, where we recall that $\norme{\cdot}_{\mathbb H_{\overline{Z}}}$ is the norm of the RKHS $\mathbb H_{\overline{Z}}$.  From  \cite[Theorem 2.1]{van2008reproducing}, this latter definition of $\mathbb H_{\overline{Z}}$ coincides with the definition of the RKHS $\mathbb H_{\overline{Z}}$ in the statement of the lemma (from \cite[Section 2.1]{van2008reproducing}, from the covariance function of $\overline{Z}$). The two definitions of the corresponding norms also coincide (again from  \cite[Theorem 2.1]{van2008reproducing}). 
\end{proof}

\subsection{Proof of Theorem \ref{th:Th2.1}}

\begin{proof}[Proof of Theorem \ref{th:Th2.1}]
Let $C>1$ and let $(\varepsilon_{c,n})_{n \in \mathbb{N}^*}$ satisfy the conditions of the theorem. We set 
$\varepsilon_n^2=\varepsilon_{c,n}^2-\log \P(W\in \mathbb B_c)/Cn\geqslant \varepsilon_{c,n}^2$. By Lemma \ref{lem:ineg_phi_phi_c} (with the definition of $\phi_{w_0}$ in \eqref{def:w0}) and since $\phi_{c,w_0}$ is non-increasing, we have
\begin{align*}
 \phi_{w_0}(\varepsilon_n) \leqslant 
 \phi_{c,w_0}( \varepsilon_n)
 \leqslant \phi_{c,w_0}(\varepsilon_{c,n})
 \leqslant n \varepsilon_{c,n}^2
 \leqslant n \varepsilon_n^2.
\end{align*} 
Furthermore, since $\varepsilon_n^2 \geqslant \varepsilon_{c,n}^2$, we have
$e^{-C n \varepsilon_n^2 } < 1/2$.  
Hence we can apply \cite[Theorem 2.1]{van2008rates} to $W$ with $\varepsilon_n$ and $C$.
Thus, there exists a sequence of measurable sets $(B_n)_{n\in \N}$, $B_n \subset \mathbb{B}$ for all $n\in \N$, such that 
\begin{align}
\log N(3\varepsilon_n,B_n,\norme{\cdot{}}_{\infty})
&\leqslant 6 Cn\varepsilon_n^2,
\label{eq:Thm2p1:VdVVZ08:1}\\
\P(W\notin B_n)&\leqslant e^{- Cn\varepsilon_n^2},
\label{eq:Thm2p1:VdVVZ08:2}\\
\P(\norme{W-w_0}_{\infty}<2\varepsilon_n)&\geqslant e^{-n\varepsilon_n^2}.
\label{eq:Thm2p1:VdVVZ08:3}
\end{align}
Now let $B_{c,n}= B_n\cap \mathbb B_c$. Clearly, $B_{c,n}\subset B_n$. 
Furthermore, since $n \varepsilon_{c,n}^2 \to \infty$, there exists a fixed $n_C \in \mathbb{N}^*$ such that, for $n\geqslant n_C$, $\varepsilon_{c,n}\geqslant \varepsilon_n/2$.
Thus, for $n\geqslant n_C$, 
\begin{align*}
\log N( 6 \sqrt{C} \varepsilon_{c,n} ,B_{c,n},\norme{\cdot{}}_{\infty}) 
&
\leqslant 
\log N( 3 \sqrt{C}  \varepsilon_n,B_{c,n},\norme{\cdot{}}_{\infty}) & \text{since $\varepsilon_{c,n} \geqslant \varepsilon_n/2$} \\
&\leqslant 
\log N( 3 \varepsilon_n, B_n,\norme{\cdot{}}_{\infty}) 
 & \text{since $B_{c,n}\subset B_n$ and $C>1$}\\
& \leqslant 
6 C n \varepsilon_n^2 & \text{from $\eqref{eq:Thm2p1:VdVVZ08:1}$}\\
&\leqslant  24 C n\varepsilon_{c,n}^2  & \text{since $\varepsilon_{c,n} \geqslant \varepsilon_n/2$}\\
&\leqslant
n  (12 \sqrt{C} \varepsilon_{c,n})^2.
\end{align*}
Hence \eqref{eq:Th2.1:1} holds. 
As for \eqref{eq:Th2.1:2}, from \eqref{eq:Thm2p1:VdVVZ08:2}, one gets
\begin{align*}
\P(W_c\notin B_{c,n})
=\P(W_c\notin  B_n)
=\P(W\notin  B_n\vert W\in \mathbb B_c)
\leqslant 
\frac{e^{-Cn\varepsilon_n^2}}{\P(W\in \mathbb B_c)}
=
e^{-Cn\varepsilon_{c,n}^2}
\end{align*}
so that \eqref{eq:Th2.1:2} holds. Last, one has
\begin{align*}
\P(\norme{W_c-w_0}_{\infty}<2\varepsilon_{c,n})
&\geqslant \P(\norme{W -w_0}_{\infty}<2\varepsilon_{c,n},\  W\in \mathbb B_c)
\geqslant e^{-\phi_{c,w_0}(\varepsilon_{c,n})}\geqslant e^{-n\varepsilon_{c,n}^2}
\end{align*}
so that  \eqref{eq:Th2.1:3} holds.
\end{proof}

\subsection{Proof of Theorem \ref{th:Th3.1}}

Before proving Theorem \ref{th:Th3.1}, let us establish the following lemma and define 
\[
K_{\max} = \max_{h=2 , \dots, H}
\max_{i=1 , \dots , d_{h+1}}
\max_{j=1 ,\dots , d_h}
K_{h,i,j}.
\]
Recall the definition of $\psi\mapsto C_{\psi}$ in \eqref{def:Cpsi}.

\begin{lemma} \label{lem:norme:infini:composition}
For $w , z \in \mathbb{B}_c$, we have
\[
\norme{ C_w - C_z }_{\infty}
\leqslant 
K_H
\norme{w - z}_{\infty}
\] 
where
$
K_H = \left( 1+ K_{\max} d_{\max} \right)^H. 
$
\end{lemma}

\begin{proof}[Proof of Lemma \ref{lem:norme:infini:composition}]

Let, for $h = 1, \dots, H$, $C_{w,h} = (P_{h,1}(w) , \dots , P_{h,d_{h+1}}(w) )$ and $C_{z,h} = (P_{h,1}(z) , \dots , P_{h,d_{h+1}}(z) )$. Then we have, for $t \in [-1,1]^d$,  recalling that $C_{w,H}$ is real-valued and letting $\nabla f$ be the gradient of a function $f$,
\begin{align*}
&
	\left| 
C_w(t) - C_z(t)	
	\right|
	\\
=
	&
	\left| 
	C_{w,H} \circ C_{w,H-1} \circ \dots \circ C_{w,1}(t)
	-
	C_{z,H} \circ C_{z,H-1} \circ \dots \circ C_{z,1}(t)
	\right| 
	\\
	 \leqslant 
	 &
		\left| 
	C_{w,H} \circ C_{w,H-1} \circ \dots \circ C_{w,1}(t)
	-
	C_{z,H} \circ C_{w,H-1} \circ \dots \circ C_{w,1}(t)
	\right| 
	\\
	& + 
		\left| 
	C_{z,H} \circ C_{w,H-1} \circ \dots \circ C_{w,1}(t)
	-
	C_{z,H} \circ C_{z,H-1} \circ \dots \circ C_{z,1}(t)
	\right| 
	\\
	\leqslant &
	\norme{      C_{w,H}  -  	C_{z,H} \   }_{\infty}
	+
		\sup_{x\in [-1,1]^{d_H}} \norme{\nabla C_{z,H}(x)}\times
		\norme{ 
		C_{w,H-1} \circ \dots \circ C_{w,1}(t) -  C_{z,H-1} \circ \dots \circ C_{z,1}(t)
		} 
		\\
		\leqslant &
		\norme{ w - z }_{\infty} 
		+
		K_{\max} 
		d_H 
		\max_{i=1,\dots,d_{ H}}
		\left| 
		P_{H-1,i}(w)
		\circ  
\dots \circ  
		C_{w,1} (t)
		-
		P_{H-1,i}(z)
	\circ 
	C_{z,H-2}
	\circ  \dots \circ  
	C_{z,1} (t)	
		\right| 
		\\
		\leqslant
		& 
		\left( 
		1  +  K_{\max} {d_{\max}} 
		\right)\times\max \Bigg(
		\norme{ w - z }_{\infty}, \\
		&  \qquad \qquad
			\max_{i=1 , \dots,d_H}
		\left| 
		P_{H-1,i}(w)
		\circ 
		C_{w,H-2}
		\circ  \dots \circ  
		C_{w,1} (t)
		-
		P_{H-1,i}(z)
		\circ 
		C_{z,H-2}
		\circ  \dots \circ  
		C_{z,1} (t)	
		\right| 
		\Bigg).
\end{align*}
Hence, we can obtain the lemma by a descending induction from $H$ to $1$. 	
\end{proof}

Now we define the Kullback-Leibler divergence: for $f,g : [-1,1]^d \to (0,\infty)$ with $\int_{[-1,1]^d} f(t) dt = 
\int_{[-1,1]^d} g(t) dt = 1$, let 
 \begin{align}\label{def:K}
  K(f,g)=\int_{[-1,1]^d}\log(f(t)/g(t))f(t)dt,
 \end{align}
 
if $\int_{[-1,1]^d}\abs{\log(f(t)/g(t))}f(t)dt<+\infty$ and $K(f,g)=+\infty$ else. We also introduce 
\begin{align}\label{def:V}
V(f,g)=\int_{[-1,1]^d}(\log(f(t)/g(t)))^2 f(t) dt.
\end{align}

Then we have the following lemma, that follows from Lemma \ref{lem:norme:infini:composition} and \cite[Lemma 3.1]{van2008rates}.

\begin{lemma}\label{lem:Lem3.1}
	Recall the Hellinger distance $h$ from \eqref{def:hellinger}, the function $\psi \mapsto C_{\psi}$ from \eqref{def:Cpsi}, and the function $z \mapsto p_z$ from \eqref{def:dens}.
For any functions $v$ and $w$ in $\mathbb B_c$, we have 
\begin{align*}
h(p_{C_v},p_{C_w})&\leqslant K_H\norme{v-w}_{\infty}e^{K_H\norme{v-w}_{\infty}/2},\\
K(p_{C_v},p_{C_w})&\leqslant c K_H^2\norme{v-w}_{\infty}^2(1+K_H\norme{v-w}_{\infty}) e^{K_H\norme{v-w}_{\infty}}, 
~ ~ ~ ~ \text{and} \\
V(p_{C_v},p_{C_w})&\leqslant c K_H^2\norme{v-w}_{\infty}^2 (1+K_H\norme{v-w}_{\infty})^2 e^{K_H\norme{v-w}_{\infty}},
\end{align*}
where $K_H$ is as in Lemma \ref{lem:norme:infini:composition} and $c \geqslant 1$ is a finite constant.
\end{lemma}

\begin{proof}[Proof of Lemma \ref{lem:Lem3.1}]
By \cite[Lemma 3.1]{van2008rates}, one has, with a finite constant $c \geqslant 1$, 
\begin{align}
h(p_{C_v},p_{C_w}) &
 \leqslant \norme{C_v-C_w}_{\infty}e^{\norme{C_v-C_w}_{\infty}/2}, \label{eq:for:lemma:one} \\
 K(p_{C_v},p_{C_w}) &
 \leqslant c \norme{C_v-C_w}_{\infty}^2
\left(
1 + \norme{C_v-C_w}_{\infty}
\right)
 e^{\norme{C_v-C_w}_{\infty}},\label{eq:for:lemma:two}  \\
 V(p_{C_v},p_{C_w}) &
 \leqslant c \norme{C_v-C_w}_{\infty}^2
\left(
1 + \norme{C_v-C_w}_{\infty}
\right)^2
 e^{\norme{C_v-C_w}_{\infty}}. \label{eq:for:lemma:three} 
\end{align}
Thus we conclude the proof from Lemma \ref{lem:norme:infini:composition}.
\end{proof}

Now we are able to prove Theorem \ref{th:Th3.1}.

\begin{proof}[Proof of Theorem \ref{th:Th3.1}]
Consider a sequence $(\varepsilon_{c,n})_{n\in \N}$ satisfying the conditions of the theorem.
Define $w_0 \in \mathbb{B}$ by, for $(h,i) \in \mathcal{I}$ and for $t \in [-1,1]^{d_{\max}}$, $w_0(t,h,i) = z_{0,h,i}(T_h(t))$. 
 Then, since $\varepsilon_{c,n}\to 0$ for $n$ large enough, from the assumptions of the theorem on the functions $z_{0,h,i}$, the conditions \eqref{eq:analyzing:phi:cond:un} and \eqref{eq:analyzing:phi:cond:deux} of Theorem \ref{thm:analyzing:phi} hold. Then Theorem \ref{thm:analyzing:phi} provides 
$\phi_{c,w_0}(\varepsilon_{c,n})\leqslant \Phi_{c,z_0}(\varepsilon_{c,n})$ and thus
\[
\phi_{c,w_0}(\varepsilon_{c,n})\leqslant n\varepsilon_{c,n}^2.
\]
Let, with $c$ and $K_H$ as in Lemma \ref{lem:Lem3.1}, 
\[
t = 2 c K_H^2,
~ ~ ~
C_c = 1 + 16t
~ ~ ~
\text{and} 
~ ~ ~
C 
=
\frac{1}{4t}.
\]
Hence,
\begin{equation} \label{eq:equal:Cp4}
\frac{C_c}{4t} = C + 4. 
\end{equation}

Let $\overline{\varepsilon}_n = 2 \sqrt{t} \varepsilon_{c,n}$ and $\varepsilon_{n}  = 14 \sqrt{C_c} K_H \varepsilon_{c,n} \geqslant \overline{\varepsilon}_n$.
Let us apply \cite[Theorem 8.9]{ghosal2017fundamentals}, with the constant $C$, the sequences $(\overline{\varepsilon}_n)_{n\in \N}$ and $(\varepsilon_{n})_{n\in \N}$, and the prior $p_{C_{W_c}}$. 
Let us thus check that the assumptions of \cite[ Theorem 8.9]{ghosal2017fundamentals} are satisfied, using Theorem \ref{th:Th2.1}.

Since $n \varepsilon_{c,n}^2 \to \infty$, there exists $N' \in \mathbb{N}^*$ such that, for $n \geqslant N'$, $e^{- C_c n \varepsilon_{c,n}^2} < 1/2$.  
 Hence, the assumptions of Theorem \ref{th:Th2.1} hold, with $C$ there given by $C_c$ and with $(\varepsilon_{c,n})_{n \geqslant N'}$ as considered above. Hence, from this theorem, there exists an integer $n_C \geqslant 0$ and a  sequence of sets $(B_{c,n})_{n \geqslant n_C}$ such that
the conclusions \eqref{eq:Th2.1:1}, \eqref{eq:Th2.1:2}, and \eqref{eq:Th2.1:3} hold (with $C$ replaced by $C_c$).
Then, in  \cite[Theorem 8.9]{ghosal2017fundamentals}, we choose $\mathcal P_{n,1} = \{p_{C_w};\ w\in B_{c,n}\}$.

\qquad $\bullet$ First, let $w_1,\dots,w_N$ be a $\norme{\cdot{}}_{\infty}$-distance covering of $B_{c,n}$ with radius $6 \sqrt{C_c} \varepsilon_{c,n}$ with $N = N(  6 \sqrt{C_c} \varepsilon_{c,n} , B_{c,n} , \|\cdot{}\|_{\infty})$.
Now let $x\in \mathcal P_{n,1}$ (so that $x=p_{C_w}$ for some $w\in B_{c,n}$). 
There is $i \in \{  1 , \dots , N  \}$ such that $\norme{ w - w_i }_{\infty} \leqslant 6 \sqrt{C_c} \varepsilon_{c,n}$. 
Then, by Lemma \ref{lem:Lem3.1}, one has
\begin{align*}
h(p_{C_w},p_{C_{w_i}}) \leqslant K_H\norme{w-w_i}_{\infty}e^{K_H\norme{w-w_i}_{\infty}/2}\leqslant 6K_H \sqrt{C_c} \varepsilon_{c,n} e^{3 K_H \sqrt{C_c} \varepsilon_{c,n}}\leqslant 
7 K_H \sqrt{C_c} \varepsilon_{c,n},
\end{align*}
for $n$ large enough, since $\varepsilon_{c,n} \to 0$. This leads to, for $n$ large enough,
\begin{align*}
 \log N\left(7 \sqrt{C_c} K_H \varepsilon_{c,n} ,\mathcal P_{n,1},h\right) 
&
\leqslant 
\log N\left(6 \sqrt{C_c} \varepsilon_{c,n} , B_{c,n} , \norme{\cdot{}}_{\infty}\right)
\\
&
\leqslant 
n (12 \sqrt{C_c} \varepsilon_{c,n})^2 & \text{from \eqref{eq:Th2.1:1} with $C$ replaced by $C_c$   }\nonumber \\
&\leqslant 
n (14 \sqrt{C_c} K_H \varepsilon_{c,n})^2 & \text{since $K_H>1$}.
\end{align*}
Hence, we obtain, for $n$ large enough,
\begin{equation} \label{eq:Ncov}
\log N\left(\varepsilon_{n} /2 ,\mathcal P_{n,1},h\right) 
\leqslant 
n  \varepsilon_{n}^2.
\end{equation}

\qquad $\bullet$ Second, using \eqref{eq:Th2.1:2} in Theorem \ref{th:Th2.1} (with $C$ replaced by $C_c$) and then \eqref{eq:equal:Cp4},
\begin{align}
\mathbb{P}( p_{C_{W_c}} \not  \in  \mathcal P_{n,1})&=\P(W_c\notin B_{c,n})\leqslant e^{-C_c n\varepsilon_{c,n}^2}
=  e^{-(C+4) n \overline{\varepsilon}_{n}^2}.\label{eq:hyp2}
\end{align}

\qquad $\bullet$ Third, for $p : [-1,1]^d \to (0,\infty)$ with $\int_{[-1,1]^d} p(x) dx = 1$, let, with  $K$ and $V$ defined in \eqref{def:K} and \eqref{def:V} respectively,
\[
V_{2,0}(p_0;p)
=
\int_{[-1,1]^d}
\left(
\log \left(  \frac{p_0(t)}{p(t)}  \right)
-
K(p_0,p)
\right)^2
p_0(t) dt
\leqslant V(p_0,p).
\]
Then, using Lemma \ref{lem:Lem3.1} and \eqref{eq:Th2.1:3} in Theorem \ref{th:Th2.1}, one has, for $n$ large enough, 
\begin{align}
&\mathbb{P}\left(
\ K(p_0,p_{C_{W_c}}) < \overline{\varepsilon}_{n}^2, \, V_{2,0}(p_0;p_{C_{W_c}}) < \overline{\varepsilon}_{n}^2
\right)\nonumber\\
&= \mathbb{P}\left(
\ K(p_0,p_{C_{W_c}}) < 4 t \varepsilon_{c,n}^2, \, V_{2,0}(p_0;p_{C_{W_c}}) < 4 t \varepsilon_{c,n}^2
\right)\nonumber\\
&\geqslant
\mathbb{P}\left(
\ K(p_0,p_{C_{W_c}}) < 4 t \varepsilon_{c,n}^2, \, V(p_0,p_{C_{W_c}}) < 4 t \varepsilon_{c,n}^2 
\right)
\notag\\
&\geqslant \mathbb{P}\Bigl( c K_H^2\norme{W_c-w_0}_{\infty}^2(1+K_H\norme{W_c-w_0}_{\infty}) e^{K_H\norme{W_c-w_0}_{\infty}} < 4 t \varepsilon_{c,n}^2,\nonumber\\
&\qquad \qquad \,  c K_H^2\norme{W_c-w_0}_{\infty}^2 (1+K_H\norme{W_c-w_0}_{\infty})^2 e^{K_H\norme{W_c-w_0}_{\infty}}  < 4 t \varepsilon_{c,n}^2 \Bigr)\nonumber\\
&\geqslant \mathbb{P}\Bigl( 2c K_H^2\norme{W_c-w_0}_{\infty}^2 < 4 t \varepsilon_{c,n}^2\Bigr)\nonumber\\
&= \mathbb{P}\Bigl( \norme{W_c-w_0}_{\infty} < 2\varepsilon_{c,n} \nonumber \Bigr)\\
&\geqslant \exp\{-n\varepsilon_{c,n}^2\} 
= \exp\left\{-n \frac{1}{4t} \overline{\varepsilon}_{n}^2  \right\}
=\exp\left\{-n C \overline{\varepsilon}_{n}^2  \right\}. \label{eq:hyp3}
\end{align}
Hence, we can apply \cite[Theorem 8.9]{ghosal2017fundamentals}, with the sequences $(\overline{\varepsilon}_n)$ and $(\varepsilon_n)$ (starting at a $n$ large enough) and the constant $C$. Indeed, \eqref{eq:Ncov} (resp. \eqref{eq:hyp2} and \eqref{eq:hyp3}) here enables to show Equation (8.5) (resp. Equation (8.6) and Equation (8.4)) in  \cite[Theorem 8.9]{ghosal2017fundamentals}. 
As a consequence, \cite[Theorem 8.9]{ghosal2017fundamentals} implies 
\eqref{eq:posterior:rate:density}.
\end{proof}

\subsection{Proof of Theorem \ref{th:Th3.2}}

With $\mathcal{L}_U$ the law of $U$ and $\mu$ the counting measure on $\{ 0,1 \}$, for any $L \colon [-1,1]^d \times \{ 0 , 1 \} \to \R$, we write $\norme{L}_{2,U}$ for the $L^2$-norm of $L$ with respect to the measure $ \mathcal{L}_U \otimes \mu $.
For $w \colon [-1,1]^d \to \R$, we let $L_w \colon [-1,1]^d \times \{  0,1 \} \to \R$ be defined by 
\begin{equation} \label{eq:p:w:classif}
L_w(u,v) = f_w(u)^v (1 - f_w )(u)^{1-v} \quad \text{for $(u,v) \in  [-1,1]^d \times \{ 0 , 1 \}$}.
\end{equation}
 Notice that, if $f_w$ (defined in \eqref{def:classif}) is a candidate function for $f_0$, then $L_w(U,V)$ is the likelihood function of $(U,V)$ with respect to the measure $ \mathcal{L}_U \otimes \mu $. 
Extend the definition of the Kullback-Liebler divergence $K$ in \eqref{def:K} and that of $V$ in \eqref{def:V}, replacing the input space $[-1,1]^d$ by $[-1,1]^d \times \{0,1 \}$ and Lebesgue measure $dt$ by $ \mathcal{L}_U \otimes \mu $. For $f : [-1,1]^d \to \mathbb{R}$, write also $\norme{f}_{2,U}$ for the $L^2$ norm of $f$ with respect to $\mathcal{L}_U$.

\begin{lemma}\label{lem:Lem3.2:classif}
	For any functions $v$ and $w$ in $\mathbb B_c$, we have 
	\begin{align*}
		\norme{L_{C_v} - L_{C_w}}_{2,U} 
		& = \sqrt{2}  \norme{f_{C_v} - f_{C_w}}_{2,U}
		 \leqslant \sqrt{2} K_H \norme{ \Psi' }_{\infty}  \norme{  v - w }_{\infty}   \\
		K(L_{C_w},L_{C_{w_0}})&\leqslant  K_H^2 
		\left( \norme{ \frac{\Psi'}{\Psi (1- \Psi)} }_{\infty}  \vee  1 \right)   \norme{w-w_0}_{\infty}^2, 
		~ ~ ~ ~ \text{and} \\
		V(L_{C_w},L_{C_{w_0}})&\leqslant  K_H^2 
		\left( \norme{ \frac{\Psi'}{\Psi (1- \Psi)} }_{\infty}  \vee  1 \right)^2   \norme{w-w_0}_{\infty}^2,
	\end{align*}
	using the definitions  \eqref{def:Cpsi}, \eqref{def:classif}, and \eqref{eq:p:w:classif} and where $K_H$ is as in Lemma \ref{lem:norme:infini:composition}.
\end{lemma}

\begin{proof}
	Lemma 3.2 of \cite{van2008rates}, with the observation that $S(w,w_0)$ there is bounded by 
	\[
		\left( \norme{ \frac{\Psi'}{\Psi (1- \Psi)} }_{\infty}  \vee  1 \right)
	\]
	yields 
		\begin{align*}
		\norme{L_{C_v} - L_{C_w}}_{2,U} 
		& = \sqrt{2}  \norme{f_{C_v} - f_{C_w}}_{2,U}
		\leqslant \sqrt{2} \norme{ \Psi' }_{\infty}  \norme{  C_v - C_w }_{2,U}   \\
		K(L_{C_w},L_{C_{w_0}})&\leqslant  
		\left( \norme{ \frac{\Psi'}{\Psi (1- \Psi)} }_{\infty}  \vee  1 \right)   \norme{C_w-C_{w_0}}_{2,U}^2, 
		~ ~ ~ ~ \text{and} \\
		V(L_{C_w},L_{C_{w_0}})&\leqslant  
		\left( \norme{ \frac{\Psi'}{\Psi (1- \Psi)} }_{\infty}  \vee  1 \right)^2   \norme{C_w-C_{w_0}}_{2,U}^2.
	\end{align*}
The proof is concluded using that $\norme{ \cdot }_{2,U} \leqslant \norme{  \cdot }_{\infty}$
and that $\norme{  C_v - C_w }_{\infty} \leqslant K_H \norme{  v - w }_{\infty}$ for $v,w \in \mathbb{B}_c $ from Lemma \ref{lem:norme:infini:composition}. 
\end{proof}

\begin{proof}[Proof of Theorem \ref{th:Th3.2}]
We proceed as in the proof of Theorem \ref{th:Th3.1}, using  \cite[Theorem 8.9]{ghosal2017fundamentals}, with $d$ there given by 	$ d(L_{C_v},L_{C_w}) =  \norme{L_{C_v} - L_{C_w}}_{2,U} / 2$ for $v$ and $w \in \mathbb{B}_c$. Remark that this choice of $d$ indeed satisfies the assumption (8.2) in   \cite{ghosal2017fundamentals}, because of the item (vi) in \cite[Lemma B.1]{ghosal2017fundamentals} (see the paragraph after the assumption (8.2) in \cite{ghosal2017fundamentals}). Remark that instead of using Lemma \ref{lem:Lem3.1} as in the proof of Theorem \ref{th:Th3.1}, we use Lemma \ref{lem:Lem3.2:classif} here. This concludes the proof by also noting that $\norme{L_{C_w} - L_{z_{0,H} \circ \dots \circ z_{0,1}}}_{2,U} 
 = \sqrt{2}  \norme{f_{ C_w}- f_{z_{0,H} \circ \dots \circ z_{0,1}}}_{2,U}
 =\sqrt{2}  \norme{f_{ C_w}- f_{0}}_{2,U}
 $ for $w  \in \mathbb{B}_c$ from Lemma \ref{lem:Lem3.2:classif}. 
\end{proof}

\subsection{Proof of Theorem \ref{theorem:rate:holder}}

Since $Z_{1},\dots,Z_H$ are univariate here, we write for simplicity, for $h = 1 , \dots,H $, $z_{0,h,1} = z_{0,h}$ (when we will apply Theorems \ref{th:Th3.1} and \ref{th:Th3.2} below), $Z_{h,1} = Z_{h}$, and $\mathbb{H}_{h,1} = \mathbb{H}_{h}$.  

For $\ell \in \N$ and for $(a,b)\in \R^2$ such that $a<b$, let $\mathcal{H}^{\ell+1}([a,b])$ be the Sobolev space  of $\ell$-times continuously differentiable functions $h$ on $[a,b]$ with $\ell$-th derivative function $h^{(\ell)}$ that is the integral of a square integrable function $h^{(\ell+1)}$ on $[a,b]$: $h^{(\ell)}(x) = \int_{a}^x h^{(\ell+1)} (t) dt $, for $ x \in [a, b]$.

For $h=1,\dots,H$, as shown in \cite[Section 10]{van2008reproducing} and with a linear change of input variables (see \cite[Lemma 7.1]{van2008reproducing}), the RKHS $\mathbb H_{h}$ of $Z_{h}$ is 
$\mathcal{H}^{N_h+1}([-1,1])$ with squared RKHS-norm, for $g \in \mathcal{H}^{N_h+1}([-1,1])$, 
\[
\norme{g}^2_{\mathbb H_{h}} = 
\int_{-1}^1 (g^{(N_h+1)}(t))^2 dt
+
\sum_{i=0}^{N_h}
g^{(i)}(-1)^2.
\]
Here $g^{(0)} = g$ and $g^{(1)} , \dots , g^{(N_h)}$ are the derivatives of $g$.

First consider the setting of Theorem \ref{th:Th3.1}.
Define $z_{0,1} , \dots , z_{0,h}$ by 
$
z_{0,1}
=
	\log(p_0) 
	/
	2
\norme{\log(p_0)}_{\infty}$,
$z_{0,h} = \mathrm{id}$ for $h = 2 , \dots , H-1$, and $z_{0,H} = 2 \norme{\log(p_0)}_{\infty} \mathrm{id}$. Here $\mathrm{id}$ is the identity function. 
Then $z_{0,H} \circ \dots \circ z_{0,1}= \log(p_0)$.
Furthermore, consider that the constants $K_1 , \dots , K_H$ are selected large enough such that the functions 
$z_{0,1} , \dots , z_{0,H}$ satisfy the conditions of Theorem \ref{th:Th3.1}. 
Recall $K_{\min}$ defined in \eqref{eq:Kmin}.

\begin{lemma} \label{lemma:approximation:error}
We have, for $\varepsilon \in (0,1]$,
\[
\inf_{
\substack{
g \in \mathbb H_{1} \\
\norme{ g - z_{0,1} } < \varepsilon
}
}
\norme{g}_{\mathbb H_{1}}^2
\leqslant 
C_{\sup}
\varepsilon^{-2( N_1 - \beta +1)/\beta}
\]
and, for $h=2 , \dots , H$, 
\[
\underset{
\substack{
g \in \mathbb H_{h} \\
\norme{ g - z_{0,h} }_{\infty} < \frac{\varepsilon}{2} \\
\norme{ \partial g / \partial x_1 - \partial z_{0,h} / \partial x_1 }_{\infty} < \frac{K_{\min}}{4}
} 
}{\inf}
\norme{g}_{\mathbb H_{h}}^2
\leqslant 
C_{\sup}.
\]
\end{lemma}

\begin{proof}[Proof of Lemma \ref{lemma:approximation:error}]

The function $z_{0,1} = \log( p_0 ) / 2 \norme{\log( p_0 )}_{\infty}$ is a function in $\mathcal{F}^{\beta}([-1,1],\mathbb{R})$ that does not depend on $\varepsilon$. Hence, because $\mathbb H_{1} = \mathcal{H}^{N_1+1}([-1,1])$ as seen above, from the proof of Theorem 4.1 in \cite{van2008rates}, we obtain the first inequality of the lemma.

To prove the second inequality, remark that $z_{0,h}$ is a linear function from $[-1,1]$ to $\mathbb{R}$ and thus it belongs to $\mathbb H_{h}$. Hence the infimum in the second inequality of the lemma is smaller than $\norme{ z_{0,h}  }^2_{\mathbb H_{h}} < \infty$.
\end{proof}

\begin{lemma} \label{lemma:concentration:term}

We have, for $\varepsilon \in (0,1]$ and for $h=1 , \dots, H$, 
\[
- \log \mathbb{P}
\left(
\norme{ Z_{h}
}_{\infty}
\leqslant  \varepsilon 
\right)
\leqslant 
C_{\sup}
\varepsilon^{-1/(N_{h} + 1/2)}.
\]
\end{lemma}

\begin{proof}[Proof of Lemma \ref{lemma:concentration:term}]
	We have, for $h=1 , \dots ,H$,
	\[
	- \log \mathbb{P}
	\left(
	\norme{ 
		Z_h
	}_{\infty} \leqslant  \varepsilon
	\right)
	=
	- \log \mathbb{P}
	\left(
	\sup_{t \in [ 0 ,2]}
	\left|
	I^{N_h} B_h( t  ) + \sum_{\ell=0}^{N_h} X_{h,\ell} \frac{t^\ell}{\ell!}
	\right|
	\leqslant  \varepsilon
	\right).
	\]
	The right-hand side above is bounded by $C_{\sup}
	\varepsilon^{-1/(N_{h} + 1/2)}$ as stated in the proof of Theorem 4.1 in \cite{van2008rates} (this follows in particular from Theorem 1.3 in \cite{chen2003quadratic}). 
\end{proof}

Recall $\Phi_{c,z_0}$ from \eqref{def:Phi_w0}. 
From Lemmas \ref{lemma:approximation:error}, \ref{lemma:concentration:term}, and \ref{lemma:proba:small:ball:non-zero} and from $\beta \leqslant N_{1}+1/2$ and $\beta \leqslant N_{h}$, for $h=2 , \dots , H$,
we obtain,  for $\varepsilon \in (0,1]$,
\[
\Phi_{c,z_0}(\varepsilon)  
\leqslant 
C_{\sup}
\varepsilon^{-2( N_{1} - \beta +1)/\beta }.
\]
Hence, without loss of generality, one may assume that $C_{\sup}>1$ and thus we have, for $n$ large enough,
\[
\Phi_{c,z_0} \left( C_{\sup} n^{ -\beta/(2 N_{1} +2)}  \right) 
\leqslant 
n
C_{\sup}^2 n^{ - 2\beta/( 2 N_1+2)}. 
\]
Hence, the conclusion of Theorem \ref{th:Th3.1} holds with $\varepsilon_{c,n}=C_{\sup} n^{ -\beta/(2 N_1 +2)}$. 
The proof for the conclusion of Theorem \ref{th:Th3.2} is the same.

\subsection{Proof of Theorem \ref{theorem:rate:riemann:liouville}}
The proof is similar to the proof of Theorem \ref{theorem:rate:holder}
As done previously, we write for simplicity, for $h = 1 , \dots,H $, $z_{0,h,1} = z_{0,h}$, $Z_{h,1} = Z_{h}$, and $\mathbb{H}_{h,1} = \mathbb{H}_{h}$.  

For any $\alpha>0$ and for any measurable function $g$ on $[-1,1]$, we define the (left-sided) Riemann-Liouville integral of $g$ of order $\alpha$ (if it exists) by, for $t \in [-1,1]$,
\[
(I_{0+}^{\alpha} g) (t)=\frac{1}{\Gamma(\alpha)}\int_{-1}^t (t-s)^{\alpha-1} g(s)ds
\] 
where $\Gamma$ stands for the standard Gamma function.
As shown in \cite[Theorem 4.2]{van2008rates} and with a linear change of input variables (see \cite[Lemma 7.1]{van2008reproducing}), using the fact that the RKHS of a sum of independent Gaussian processes is the sum of their RKHS's (see \cite[Theorem 5]{berlinet2004reproducing}), for $h=1,\dots,H$, the RKHS $\mathbb H_h$ of $Z_{h}$ is the set of functions of the form 
\begin{equation} \label{eq:RKHS:riemann:liouville}
I_{0+}^{\alpha_h+1/2}(g)
+
P_{\underline{\alpha}_h+1},
\end{equation}
where $g $ is square integrable on $[-1,1]$
and where $P_{\underline{\alpha}_h+1}$ is a polynomial of degree $\underline{\alpha}_h+1$.

\medskip

First consider the setting of Theorem \ref{th:Th3.1}. 
Define $z_{0,1} , \dots , z_{0,H}$ as in the proof of Theorem \ref{theorem:rate:holder} with in particular $z_{0,H} \circ \dots \circ z_{0,1}= \log(p_0)$. Also as in the proof of Theorem \ref{theorem:rate:holder}, let the constants $K_1 , \dots , K_H$ be selected large enough such that the functions $z_{0,1} , \dots , z_{0,H}$ satisfy the conditions of Theorem \ref{th:Th3.1}.

\begin{lemma}\label{lemma:h:one:riemann:liouville}
	We have, for $\varepsilon \in (0,1]$,
	\[
	\frac{3}{2} 	\inf_{
	\substack{
		g \in \mathbb H_{1} \\
		\norme{  g - z_{0,1} }_{\infty} < \varepsilon
	}
} 
\norme{  g }_{\mathbb H_{1}}^2
- 2 \log 
\mathbb{P}
\left(
\norme{Z_{1}}_{\infty} < \varepsilon 
\right)
\leqslant 
C_{\sup} \varepsilon^{-1 / \alpha_1}
\]
%
%
	and, for $h = 2 , \dots , H$,
	\begin{align*} 
		\frac{3}{2}
\underset{
	\substack{
		g \in \mathbb H_{h} \\
		\norme{ g - z_{0,h} }_{\infty} < \frac{\varepsilon}{2} \\
		\norme{ \partial g / \partial x_1 - \partial z_{0,h} / \partial x_1 }_{\infty} < \frac{K_{\min}}{4}, \\
	} 
}{\inf}
&
\norme{g}_{\mathbb H_{h}}^2
- 2 \log \mathbb{P}
\Bigl(
\norme{ 
	Z_{h}
}_{\infty} \leqslant  \frac{\varepsilon}{2}
\Bigr)
\\
		&
		- 2 
\log \mathbb{P}
\Bigl(
\norme{ 
	\partial Z_{h} / \partial x_1
}_{\infty}
\leqslant  \frac{K_{\min}}{4}
\Bigr)  
\leqslant C_{\sup} \varepsilon^{-1/ \alpha_h }.
	\end{align*}
\end{lemma}

\begin{proof}[Proof of Lemma \ref{lemma:h:one:riemann:liouville}]
The first part of the lemma comes with a straightforward application of \cite[Theorem 4.3]{van2008rates} (note that the arguments of its proof would be the same if the support $[0,1]$ there was replaced by $[-1,1]$).

Then the second infimum in the statement of the lemma is bounded because $z_{0,h}$ (a multiple of identity) belongs to $\mathbb H_{h}$ since $\mathbb H_{h}$ contains the polynomials of degrees up to $\underline{\alpha}_h+1$. The quantity $- 2 \log \mathbb{P}
	\Bigl(
	\norme{ 
		Z_{h}
	}_{\infty} \leqslant  \varepsilon/2
	\Bigr)$ is smaller than the term displayed in Lemma \ref{lemma:h:one:riemann:liouville} (with $Z_1$ there replaced by $Z_h$) and thus this quantity is bounded by $C_{\sup} \varepsilon^{-1/\alpha_h}$ from this lemma. Finally, the quantity $-2 \log \mathbb{P}
	\Bigl(
	\norme{ 
		\partial Z_{h} / \partial x_j
	}_{\infty}
	\leqslant K_{\min}/4
	\Bigr)$ does not depend on $\varepsilon$.
\end{proof}

Hence, we obtain from $\alpha_h \geqslant \alpha_1$, for $h=1 , \dots , H$, that $\Phi_{c,z_0}(\varepsilon) \leqslant C_{\sup} \varepsilon^{-1/\alpha_1}$. Then, as in the proof of Theorem \ref{theorem:rate:holder} and without loss of generality, one may assume that $C_{\sup}>1$ and thus we have, for $n$ large enough,
\[
\Phi_{c,z_0} \left( C_{\sup} n^{ -\alpha_1/(2 \alpha_1 +1)}  \right) 
\leqslant
n
C_{\sup}^2 n^{ - 2 \alpha_1/(2 \alpha_1 +1)}. 
\]
Hence, the conclusion of Theorem \ref{th:Th3.1} holds with $\varepsilon_{c,n}=C_{\sup} n^{ -\alpha_1/(2 \alpha_1 +1)} $. 
The proof for the conclusion of Theorem \ref{th:Th3.2} is the same.

\subsection{Proof of Theorem \ref{theorem:rate:matern}}

The proof is similar to the proof of Theorems \ref{theorem:rate:holder} and \ref{theorem:rate:riemann:liouville}. 
By \cite[Lemma 4.1]{van2009adaptive} (as in \cite[(11)]{van2011information}), the 
RKHS $\mathbb H_{h,i}$ of $Z_{h,i}$ is the set of restrictions to $[-1,1]^{d_h}$ of 
 all real parts of the functions $h : \R^{d_h} \to \R$ of the form
\begin{equation} \label{eq:condition:FT:RKHS:matern}
h(t)=\int_{\mathbb{R}^d} e^{i\lambda^\top t}\Psi(\lambda)m_{h,i}(\lambda)d\lambda
\end{equation}
for $\Psi : \mathbb{R}^d \to \mathbb{R}$ such that $\int_{\mathbb{R}^d} \Psi^2(\lambda) m_{h,i} (\lambda) d \lambda < \infty $, with $m_{h,i}$ as in \eqref{eq:Mhi}.

First consider the setting of Theorem \ref{th:Th3.1}. 
Define $z_{0,h,i}$ for $(h,i) \in \mathcal{I}$ as follows. Without loss of generality (up to swapping $Z_{1,1} , \dots , Z_{1,d_2}$), we may assume that $\alpha_{1,1} = \alpha_{1,\min}$. Then we let $z_{0,1,1} = \log(p_0) / 2 \norme{\log(p_0)}_{\infty}, z_{0,1,2} = \dots = z_{0,1,d_2} = 0$. For $h = 2 , \dots , H-1$, we let $z_{0,h,1} ( u_1 , \dots , u_{d_h}) = u_1$ for $(u_1 , \dots , u_{d_h}) \in [-1,1]^{d_h}$ and we let $z_{0,h,2} =  \dots = z_{0,h,d_{h+1}} = 0$. Finally, we let $z_{0,H,1} = ( u_1 , \dots , u_{d_H}) = 2 u_1  \norme{\log(p_0)}_{\infty}$ for $(u_1 , \dots , u_{d_H}) \in [-1,1]^{d_H}$. Then $z_{0,H} \circ \dots \circ z_{0,1}= \log(p_0)$.
Furthermore, consider that the constants $K_{h,i}$ for $(h,i) \in \mathcal{I}$ are selected large enough such that the functions $z_{0,h,i}$ for $(h,i) \in \mathcal{I}$ satisfy the conditions of Theorem \ref{th:Th3.1}.

\begin{lemma}\label{lemma:h:one:matern}
	We have, for $\varepsilon \in (0,1]$, 
	\[
	\frac{3}{2} 	\inf_{
		\substack{
			g \in \mathbb H_{1,1} \\
			\norme{  g - z_{0,1,1} }_{\infty} < \varepsilon
		}
	} 
	\norme{  g }_{\mathbb H_{1,1}}^2
	- 2 \log 
	\mathbb{P}
	\left(
	\norme{Z_{1,1}}_{\infty} < \varepsilon 
	\right)
	\leqslant 
	C_{\sup} \varepsilon^{-d / \alpha_{1,1}}
	+
	C_{\sup} \varepsilon^{
		-
		\frac{2 \alpha_{1,1} + d - 2 \beta}{\beta}
	}
	\]
%
%
	and,  for $(h,i) \in \mathcal{I} \backslash \{ (1,1) \}$, if $h=1$,
	\begin{align*} 
		\frac{3}{2}
		\underset{
			\substack{
				g \in \mathbb H_{h,i} \\
				\norme{ g -z_{0,h,i} }_{\infty} < \frac{\varepsilon}{2} 
			} 
		}{\inf} &
		\norme{g}_{\mathbb H_{h,i}}^2
		- 2 \log \mathbb{P}
		\Bigl(
		\norme{ 
			Z_{h,i}
		}_{\infty} \leqslant  \frac{\varepsilon}{2}
		\Bigr)
		\leqslant C_{\sup} \varepsilon^{-d/ \alpha_{h,i} }
	\end{align*}
	and for $h=2,\dots,H$,	
	\begin{align*} 
		\frac{3}{2}
		\underset{
			\substack{
				g \in \mathbb H_{h,i} \\
				\norme{ g -z_{0,h,i} }_{\infty} < \frac{\varepsilon}{2} \\
				\norme{ \partial g / \partial x_j - \partial z_{0,h,i} / \partial x_j }_{\infty} < \frac{K_{\min}}{4}, \\
				j = 1 , \dots , d_h
			} 
		}{\inf} &
		\norme{g}_{\mathbb H_{h,i}}^2
		- 2 \log \mathbb{P}
		\Bigl(
		\norme{ 
			Z_{h,i}
		}_{\infty} \leqslant  \frac{\varepsilon}{2}
		\Bigr)
		\\
		&
		- 2 
		\log \mathbb{P}
		\Bigl(
		\norme{ 
			\partial Z_{h,i} / \partial x_j
		}_{\infty}
		\leqslant  \frac{K_{\min}}{4}
		\Bigr) 
		\leqslant C_{\sup} \varepsilon^{-d/ \alpha_{h,i} }.
	\end{align*}
\end{lemma}

\begin{proof}[Proof of Lemma \ref{lemma:h:one:matern}]
A straightforward application of \cite[Lemmas 3 and 4]{van2011information} leads to the first part of the lemma.

Then the two last infima in the statement of the lemma are bounded because $z_{0,h,i}$ (either zero or a linear function) belongs to $\mathbb H_{h,i}$. Indeed, $z_{0,h,i}$, defined on $[-1,1]^{d_h}$ can be written as the restriction of an infinitely differentiable compactly supported function on $\R^{d_h}$, the latter function thus satisfying \eqref{eq:condition:FT:RKHS:matern}. 
	The quantity $- 2 \log \mathbb{P}
	\Bigl(
	\norme{ 
		Z_{h,i}
	}_{\infty} \leqslant  \varepsilon/2
	\Bigr)$ is bounded by $C_{\sup} \varepsilon^{-d/\alpha_{h,i}}$
	from \cite[Lemma 3]{van2011information}. 
Finally, the quantity $\log \mathbb{P}
	\Bigl(
	\norme{ 
		\partial Z_{h,i} / \partial x_j
	}_{\infty}
	\leqslant  K_{\min}/4	\Bigr)$ does not depend on $\varepsilon$.
\end{proof}

Hence, we obtain from $\alpha_{h,i} \geqslant \beta$ for $(h,i) \in \mathcal{I}$ that 
\begin{align*}
\Phi_{c,z_0}(\varepsilon) & \leqslant C_{\sup} \varepsilon^{-d/ \min_{(h,i) \in \mathcal{I}} \alpha_{h,i}}
+ C_{\sup} 
\varepsilon^{-( 2 \alpha_{1,1} + d - 2 \beta )/ \beta} \leqslant 
C_{\sup} 
\varepsilon^{-( 2 \alpha_{1,1} + d - 2 \beta )/ \beta}.
\end{align*}
Hence, as in the proof of Theorems \ref{theorem:rate:holder} and \ref{theorem:rate:riemann:liouville}, without loss of generality, one may assume that $C_{\sup}>1$ and thus we have
\[
\Phi_{c,z_0} \left(
C_{\sup} 
n^{-  \beta /
	(2 \alpha_{1,\min}  + d)}
  \right) 
\leqslant
n
C_{\sup}^2 n^{ - 2 \beta /
	(2 \alpha_{1,\min} +d)}. 
\]
Hence, the conclusion of Theorem \ref{th:Th3.1} holds with $\varepsilon_{c,n}=C_{\sup} 
n^{-  \beta /
	(2 \alpha_{1,\min}  + d)}$. 
The proof for the conclusion of Theorem \ref{th:Th3.2} is the same.

\paragraph{Acknowledgments}

We are very grateful to Jean-Marc Azaïs, Ismaël Castillo, Michel Ledoux and Thibault Randrianarisoa for fruitful discussions, technical advices and constructive comments. We acknowledge support from the ANR, with the project GAP.

\bibliographystyle{abbrv}
\bibliography{biblio_proc_gauss} 
\end{document}